\documentclass{article}

\usepackage[textwidth=125mm, textheight=195 mm]{geometry}
\usepackage{amsmath,amsthm,latexsym,amssymb,wasysym,array}
\usepackage{graphicx}
\usepackage{subfigure}
\usepackage{authblk}
\usepackage{cancel}
\usepackage{bbm}
\usepackage{xargs}
\usepackage{enumerate}
\usepackage{enumitem}
\setlist[enumerate]{label={\upshape(\roman*)}}
\usepackage{multirow} 
\usepackage[flushleft]{threeparttable}
\usepackage{rotating}

\usepackage[utf8]{inputenc} 
\usepackage[T1]{fontenc}    
\usepackage{hyperref}       
\hypersetup{colorlinks,linkcolor={blue},citecolor={red},urlcolor={blue}}
\usepackage{url}            
\usepackage{booktabs}       
\usepackage{amsfonts}       
\usepackage{nicefrac}       
\usepackage{microtype}      
\usepackage{xcolor}

\usepackage[numbers]{natbib}

\newcommand{\floor}[1]{{\lfloor #1 \rfloor}}
\newcommand{\ceil}[1]{{\lceil #1 \rceil}}

\newtheorem{theorem}{Theorem}[section]
\newtheorem{proposition}[theorem]{Proposition}
\newtheorem{lemma}[theorem]{Lemma}
\newtheorem{corollary}[theorem]{Corollary}
\newtheorem{remark}[theorem]{Remark}
\newtheorem{definition}[theorem]{Definition}

\newtheorem{assumption}{Assumption}

\def\nset{\mathbb{N}}
\def\rset{\mathbb{R}}
\def\rmd{\mathrm{d}}
\def\E{\mathbb{E}}
\def\rme{\mathrm{e}}
\def\bR{\mathbb{R}}
\def\R{\mathbb{R}}
\def\calF{\mathcal{F}}
\def\calL{\mathcal{L}}
\DeclareMathOperator{\KL}{KL}
\DeclareMathOperator{\Var}{Var}
\DeclareMathOperator{\Cov}{Cov}

\newcommandx{\CPE}[3][1=]{{\mathbb E}^{#1}\left[\left. #2 \, \right| #3 \right]} 
\allowdisplaybreaks

\begin{document}

\title{Nonasymptotic estimates for Stochastic Gradient Langevin Dynamics under local conditions in nonconvex optimization
\thanks{This work was supported by The Alan Turing Institute for Data Science and AI under EPSRC grant EP/N510129/1. Y. Z. was supported by The Maxwell Institute Graduate School in Analysis and its Applications, a Centre for Doctoral Training funded by the UK Engineering and Physical Sciences Research Council (grant EP/L016508/01), the Scottish Funding Council, Heriot-Watt University and the University of Edinburgh. \"{O}. D. A. is supported by the Lloyd’s Register Foundation Data Centric Engineering Programme and EPSRC Programme Grant EP/R034710/1. T. D. acknowledges support from EPSRC EP/T004134/1, UKRI Turing AI Fellowship EP/V02678X/1, and Lloyd’s Register Foundation programme on Data Centric Engineering through the London Air Quality project.}}

\author[1]{Ying Zhang \thanks{Corresponding author. Email: ying.zhang@ntu.edu.sg}}
\author[2]{\"{O}mer Deniz Akyildiz \thanks{Email: deniz.akyildiz@imperial.ac.uk}}
\author[3,4]{Theodoros Damoulas \thanks{Email: t.damoulas@warwick.ac.uk}}
\author[3,5,6]{Sotirios Sabanis \thanks{Email: s.sabanis@ed.ac.uk}}

\affil[1]{\footnotesize Nanyang Technological University, Singapore}
\affil[2]{\footnotesize Imperial College London, UK.}
\affil[3]{\footnotesize The Alan Turing Institute, UK.}
\affil[4]{\footnotesize The University of Warwick, UK.}
\affil[5]{\footnotesize The University of Edinburgh, UK.}
\affil[6]{\footnotesize National Technical University of Athens, Greece.}

\date{}

\maketitle

\begin{abstract}
In this paper, we are concerned with a non-asymptotic analysis of sampling algorithms used in nonconvex optimization. In particular, we obtain non-asymptotic estimates in Wasserstein-1 and Wasserstein-2 distances for a popular class of algorithms called Stochastic Gradient Langevin Dynamics (SGLD). In addition, the aforementioned Wasserstein-2 convergence result can be applied to establish a non-asymptotic error bound for the expected excess risk. Crucially, these results are obtained under a local Lipschitz condition and a local dissipativity condition where we remove the uniform dependence in the data stream. We illustrate the importance of this relaxation by presenting examples from variational inference and from index tracking optimization. 
\end{abstract}


\section{Introduction}
We consider a nonconvex stochastic optimization problem
\[
\text{minimize} \quad U(\theta) : = \mathbb{E}[f(\theta, X)],
\]
where $\theta \in \mathbb{R}^d$ and $X$ is a random element. We aim to generate an estimate $\hat{\theta}$ such that the expected excess risk $\mathbb{E}[U(\hat{\theta})] - \inf_{\theta \in \mathbb{R}^d} U(\theta)$ is minimized. The optimization problem of minimizing $U$ is closely linked to the problem of sampling from a target distribution which concentrates around the minimizers of $U$. It is, therefore, important to investigate the Langevin dynamics based algorithms and their sampling behaviour in the context of optimization. The latter is the primary focus of this article.

The Langevin SDE is given by
\begin{equation}\label{sdeintro}
\rmd Z_t = -  h(Z_t) \rmd t + \sqrt{2 \beta^{-1}} \rmd B_t, \qquad t> 0,
\end{equation}
with a (possibly random) initial condition $\theta_0$, where $h: = \nabla U$, $\beta>0$, and $(B_t)_{t \ge 0}$ is a $d$-dimensional Brownian motion. Under mild conditions, it is well-known that SDE \eqref{sdeintro} admits as a unique invariant measure $ \pi_{\beta}( \theta) \wasypropto \exp(-\beta U(\theta))$. Moreover, $\pi_{\beta}$ concentrates around the minimizers of $U$ when $\beta$ takes sufficiently large values (see, e.g., \cite{hwang}). 
To sample from $\pi_{\beta}$, a standard approach is to approximate the Langevin SDE \eqref{sdeintro} by using an Euler discretization scheme, which serves as a sampling algorithm and is known as the unadjusted Langevin algorithm (ULA) or Langevin Monte Carlo (LMC). Theoretical guarantees for the convergence of ULA in Wasserstein distance and in total variation have been obtained under the assumption that $U$ is strongly convex with globally Lipschitz gradient \cite{dalalyan,unadjusted,aew}. Extensions which include locally Lipschitz gradient and higher order algorithms can be found in \cite{tula}, \cite{dk} and \cite{hola}.

In practice, however, the gradient $h$ is usually unknown and one only has an unbiased estimate of $h$. A natural extension of ULA, which was introduced in \cite{wt} in the context of Bayesian inference and which has found great applicability in this type of stochastic optimization problems, is the Stochastic Gradient Langevin Dynamics (SGLD) algorithm. More precisely, fix an $\mathbb{R}^d$-valued random variable $\theta_0$ representing its initial value and let $(X_n)_{n \in \mathbb{N}}$ be an i.i.d. sequence, the SGLD algorithm corresponding to SDE \eqref{sdeintro} is given by, for any $n\in\mathbb{N}$,
\begin{equation}\label{eq:discreteTimeSGLDintro}
\theta^{\lambda}_0:=\theta_0,\quad \theta^{\lambda}_{n+1}:=\theta^{\lambda}_n-\lambda H(\theta^{\lambda}_n,X_{n+1})+\sqrt{2\lambda\beta^{-1}}\xi_{n+1},
\end{equation}
where $\lambda >0$ is often called the stepsize or gain of the algorithm, $\beta>0$ is the so-called inverse temperature parameter, $H:\mathbb{R}^d\times\mathbb{R}^m\to\mathbb{R}^d$ is a measurable function 
and $(\xi_n)_{n\in\mathbb{N}}$ is an independent sequence of standard $d$-dimensional Gaussian random variables. The properties of the i.i.d. process $(X_n)_{n \in \mathbb{N}}$ are given below.

For a strongly convex objective function $U$, \cite{convex}, \cite{ppbdm}, \cite{pmlrv65dalalyan17a}, and \cite{dk} obtain non-asymptotic bounds in Wasserstein-2 distance between the SGLD algorithm and the target distribution $\pi_{\beta}$. While \cite{dk} assumes the stochastic gradient $H$ is a linear combination of $h$ and $(X_n)_{n \in \mathbb{N}}$, which allows bounded conditional bias, a general form of $H$ with non-Markovian $(X_n)_{n \in \mathbb{N}}$ is considered in \cite{convex}. For the case where $U$ is nonconvex, 
one line of research is to consider a dissipativity condition. The first such non-asymptotic estimate is provided by \cite{raginsky} in Wasserstein-2 distance although its rate of convergence is $\lambda^{5/4}n$ which depends on the number of iterations $n$. Improved results are obtained in \cite{xu}, by using a direct analysis of the ergodicity of the overdamped Langevin Monte Carlo algorithms. While a faster convergence rate is achieved in \cite{xu} compared to \cite{raginsky}, it is still dependent on $n$. Recently, \cite{nonconvex} obtained a convergence rate 1/2 in Wasserstein-1 distance. 
Its analysis relies on the construction of certain auxiliary continuous processes and the contraction results in \cite{eberle}. Another line of research is to assume a convexity at infinity condition of $U$. \cite{berkeley} and \cite{alex} obtain convergence results in Wasserstein-1 distance by using the contraction property developed in \cite{eberleold}. In both convex and nonconvex settings, the non-asymptotic analysis of the Langevin diffusion can be extended to a wider class of diffusions under certain conditions, see \cite{ems} and references therein.


In this paper, we establish non-asymptotic convergence results in Theorem \ref{main} and Corollary \ref{cw2} for the SGLD algorithm \eqref{eq:discreteTimeSGLDintro} in Wasserstein-1 and Wasserstein-2 distances, respectively. Moreover, by using a similar splitting approach as in \cite{raginsky}, the Wasserstein-2 convergence result can then be applied to establish a nonasymptotic error bound for the expected excess risk, which is provided in Corollary \ref{eer}. These main results are obtained under the relaxed conditions as stated in Assumptions \ref{loclip} and \ref{assum:dissipativity} below. Crucially, we relax substantially the assumptions of dissipativity and Lipschitz continuity on the stochastic gradient $H(\theta, x)$ by allowing non-uniform dependence in $x$. 

To illustrate the applicability of the proposed algorithm under the local assumptions, examples from variational inference (VI) and from index tracking optimization are considered, 
which represent key paradigms in statistical machine learning and financial mathematics. In the VI example, a nonconvex objective function is considered, and it can be shown that its stochastic gradient, denoted by $H(\theta, u)$, satisfies the local dissipativity and local Lipschitz conditions. To the best of the authors' knowledge, this is the first time that non-asymptotic guarantees are provided for a concrete variational inference example due to the local nature of the aforementioned dissipativity and Lipschitz conditions which stem from the lack of a uniform bound in $u$. As for the example from index tracking optimization, the mean squared tracking error is considered as the objective function (see, e.g. \cite{zheng2020index}, \cite{gaivoronski2005optimal}). Reparametrization is performed to remove the constraints on the parameter $\theta$, which results in a nonconvex objective function. In addition, as this example can be viewed as an online regression problem, a uniform bound of the data stream is unavailable. However, one can check that the stochastic gradient, denoted by $H(\theta, z)$, satisfies the local dissipativity and local Lipschitz conditions but not the corresponding global ones.

We conclude this section by introducing some notation. Let $(\Omega,\mathcal{F},P)$ be a probability space. We denote by $\E[X]$  the expectation of a random variable $X$.
For $1\leq p<\infty$, $L^p$ is used to denote the usual space of $p$-integrable real-valued random variables. The $L^p$-integrability of a random variable $X$ is defined as $\mathbb{E}[|X|^p] < \infty$. Fix an integer $d\geq 1$. For an $\mathbb{R}^d$-valued random variable $X$, its law on $\mathcal{B}(\mathbb{R}^d)$ (the Borel sigma-algebra of $\mathbb{R}^d$) is denoted by $\mathcal{L}(X)$. For a positive real number $a$, we denote by $\floor{a}$ its integer part. For a vector $b \in \mathbb{R}^d$, denote by $b^{\mathsf{T}}$ its transpose. Scalar product is denoted by $\langle \cdot,\cdot\rangle$, with $|\cdot|$ standing for the corresponding norm. Let $f:\mathbb{R}^{d} \rightarrow \mathbb{R}$ be a twice continuously differentiable function. Denote by $\nabla f$, $\nabla^2 f$ and $\Delta f$ the gradient of $f$, the Hessian of $f$ and the Laplacian of $f$, respectively. For any integer $q \geq 1$, let $\mathcal{P}(\mathbb{R}^q)$ denote the set of probability measures on $\mathcal{B}(\mathbb{R}^q)$. 
For $\mu,\nu\in\mathcal{P}(\mathbb{R}^d)$, let $\mathcal{C}(\mu,\nu)$ denote the set of probability measures $\zeta$
on $\mathcal{B}(\mathbb{R}^{2d})$ such that its respective marginals are $\mu,\nu$. For two probability measures $\mu$ and $\nu$, the Wasserstein distance of order $p \geq 1$ is defined as, for any $ \mu,\nu\in\mathcal{P}(\rset^d)$, $
W_p(\mu,\nu):=\inf_{\zeta\in\mathcal{C}(\mu,\nu)}
\left(\int_{\rset^d}\int_{\rset^d}|\theta-\theta'|^p\zeta(\rmd \theta \rmd \theta')\right)^{1/p}$.

\section{Main results and comparisons}\label{resultcomparison}
Let $f: \R^d \times \R^m \rightarrow \R$ be a measurable function. It satisfies $\E[|f(\theta, X)|]<\infty$ for all $\theta \in \R^d$, where $X$ is a random variable with probablity law $\mathcal{L}(X)$. Let $U: \R^d \rightarrow \R$ defined by $U(\theta) := \E[f(\theta, X)]$ be a continuously differentiable function with gradient denoted by $h:=\nabla U$. Moreover, define
\begin{equation}\label{pibetaexp}
\pi_{\beta}(A) := \frac{\int_A e^{-\beta U(\theta)} \, \rmd \theta}{\int_{\R^d} e^{-\beta U(\theta)} \, \rmd \theta}, \quad A \in \mathcal{B}(\R^d),
\end{equation}
with $\int_{\R^d} e^{-\beta U(\theta)} \, \rmd \theta <\infty$.

Denote by $(\mathcal{G}_n)_{n\in\mathbb{N}}$ a given filtration representing the flow of past information, and denote by $\mathcal{G}_{\infty} := \sigma(\bigcup_{n \in \mathbb{N}} \mathcal{G}_n)$. Fix $m \geq 1$. Let $(X_n)_{n\in\mathbb{N}}$ be an $\mathbb{R}^m$-valued, $(\mathcal{G}_n)$-adapted process with $X_n\sim \mathcal{L}(X)$ for all $n \in \mathbb{N}$.  It is assumed throughout the paper that $\theta_0$, $\mathcal{G}_{\infty}$ and $(\xi_{n})_{n\in\mathbb{N}}$ are independent. Next, we introduce our main assumptions.

Fix $\beta>0$. For each $\lambda >0$, the SGLD algorithm is given by, for any $n\in\mathbb{N}$,
\begin{equation}\label{eq:discreteTimeSGLD}
\theta^{\lambda}_0:=\theta_0,\quad \theta^{\lambda}_{n+1}:=\theta^{\lambda}_n-\lambda H(\theta^{\lambda}_n,X_{n+1})+\sqrt{2\lambda\beta^{-1}}\xi_{n+1},
\end{equation}
where $H:\mathbb{R}^d\times\mathbb{R}^m\to\mathbb{R}^d$ is a measurable function and $(\xi_n)_{n\in\mathbb{N}}$ is an independent sequence of standard $d$-dimensional Gaussian random variables. 

Then, we present our assumptions. The first assumption describes the requirement on the moment of the initial parameter $\theta_0$. Moreover, it is stated that stochastic gradient $H(\theta,\cdot)$ is assumed to be unbiased.
\begin{assumption}\label{iid}
$|\theta_0| \in L^4$. The process $(X_n)_{n \in \nset}$ is i.i.d.. 
Moreover, it holds that $\E[H(\theta,X_0)]=h(\theta)$.
\end{assumption}
Our second assumption describes the requirement on the moment of the initial data $X_0$ and on the regularity of the stochastic gradient with respect to its first and second arguments. As a result, growth estimates are derived.
\begin{assumption}\label{loclip}
There exist $\eta: \mathbb{R}^m \rightarrow [1, \infty)$ with $(1+|X_0|)\eta(X_0) \in L^4$ and positive constants $L_1$, $L_2$ such that, for all $x,x'\in\mathbb{R}^m$ and $\theta, \theta'\in\mathbb{R}^d$,
\begin{align*}
 |H(\theta,x)- H(\theta',x)| & \le L_1\eta(x)|\theta-\theta'|, \\
 |H(\theta,x)- H(\theta,x')| & \le L_2(\eta(x)+\eta(x'))(1+ |\theta|)|x-x'|.
\end{align*}
\end{assumption}
\begin{remark}\label{rem:BoundsOnH} 
Assumption \ref{loclip} implies, for all $\theta,\theta' \in \rset^{d}$,
\begin{equation}\label{mulyan}
| h(\theta)-h(\theta')| \leq L_1\E[\eta(X_0)]|\theta-\theta' | \,.
\end{equation}
Also, Assumption \ref{loclip} implies
\begin{equation}\label{growthsup}
|H(\theta,x)|\leq L_1\eta(x)|\theta|+L_2\bar{\eta}(x)+H_\star,
\end{equation}
where 
$\bar{\eta}(x) = (\eta(x)+\eta(0))|x|$ and $H_\star:=|H(0,0)|$.
Moreover, under Assumptions \ref{iid} and \ref{loclip}, the gradient $h(\theta)=\E[H(\theta,X_0)]$ for all $\theta\in\mathbb{R}^{d}$, is well-defined. 
\end{remark}
The proof of the statements in Remark \ref{rem:BoundsOnH} is postponed to Appendix~\ref{proof:rem:BoundsOnH}.

Our next assumption is a \textit{dissipativity} condition for stochastic gradients. We note that this and the previous assumption significantly relax the analogous requirements found in the literature, e.g. see \cite{raginsky}, \cite{nonconvex} and references therein.
\begin{assumption}\label{assum:dissipativity}
There exist a measurable (symmetric matrix-valued) function $A:\rset^m\to\rset^{d\times d}$ and a measurable function $\hat{b}: \mathbb{R}^m \to \mathbb{R}$ such that for any $ x\in \mathbb{R}^m$, $y \in \rset^d$, $\langle y, A(x) y\rangle \geq 0$ and for all $\theta \in \mathbb{R}^d$ and $x\in\mathbb{R}^m$, 
\[
\langle H(\theta,x),\theta\rangle\geq \langle \theta, A(x) \theta\rangle -\hat{b}(x).
\] 
The smallest eigenvalue of $\E[A(X_0)]$ is a positive real number $a>0$ and $E[\hat{b}(X_0)] := b>0$.
\end{assumption}
\begin{remark}By Assumptions \ref{iid} and \ref{assum:dissipativity}, one obtains a dissipativity condition of $h$, i.e., for any $\theta\in\mathbb{R}^d$, $
\left\langle h(\theta),\theta\right\rangle\geq a |\theta|^2-b.$
\end{remark}

\begin{remark} We emphasize that we call the process $(X_n)_{n\geq 0}$ as `data', following the convention \cite{convex,nonconvex}. This can represent `data' in the classical meaning but also can represent, e.g., samples from variational approximations in the Bayesian inference setting (see Section~\ref{sec:VI}). In the latter case, its statistical properties are straightforward to assess (since the variational approximation is a design choice) and our assumptions are easier to verify as shown in Sec.~\ref{sec:VI}.
\end{remark}


We next state our main result, which fully characterises the convergence in Wasserstein-1 distance of the law of the SGLD at its $n$-th iteration, which is denoted by $\mathcal{L}(\theta_n^\lambda)$, to the target measure $\pi_\beta$. Define first
\begin{equation}
\label{eq:definition-lambda-max}
\lambda_{\max}:= \min\left\{\frac{\min\{a,a^{1/3}\}}{16 (1+L_1)^2\left(\mathbb{E}\left[(1+\eta(X_0))^4\right]\right)^{1/2}},\frac{1}{a}\right\} \,,
\end{equation}
where $L_1$ and $a$ are defined in Assumptions \ref{loclip} and \ref{assum:dissipativity}, respectively.

\begin{theorem}\label{main} Let Assumptions \ref{iid}, \ref{loclip} and \ref{assum:dissipativity} hold. Then, there exist constants $\dot{c}, C_1,C_2, C_3>0$ such that, for every $\beta>0$, $0<\lambda\leq \lambda_{\max}$, and $\ n\in\mathbb{N}$,
\begin{equation*}
W_1(\mathcal{L}(\theta^{\lambda}_n),\pi_{\beta})\leq C_1 \rme^{-\dot{c}\lambda n/2}(\mathbb{E}[|\theta_0|^4]+1) +(C_2+C_3)\sqrt{\lambda},
\end{equation*}
where $\dot{c}$ is given in \eqref{dotc}, $C_1,C_2, C_3$ are given explicitly in \eqref{mainthmconst}. Moreover, for any $\varepsilon>0$, if we choose $\lambda \leq \frac{\varepsilon^2}{4(C_2+C_3)^2} \wedge \lambda_{\max}$, and 
\[
n\geq \frac{C_{\star}e^{C_{\star}(1+d/\beta)(1+\beta)}}{\varepsilon^2\dot{c}}\left(1+\frac{1}{(1-e^{-\dot{c}})^2}\right)\ln\left( \frac{C_{\star}e^{C_{\star}(1+d/\beta)(1+\beta)}}{\varepsilon}\left(1+\frac{1}{1-e^{-\dot{c}}}\right)\right)
\]
with $C_{\star}>0$ independent of $d, \beta, n$, then $W_1(\mathcal{L}(\theta^{\lambda}_n),\pi_{\beta})\leq \varepsilon$.
\end{theorem}
By further observing a trivial functional inequality, one can state an analogous result in Wasserstein-2 distance, which matches the rate obtained in \cite{alex} but which is known to be suboptimal.
\begin{corollary}\label{cw2} Let Assumptions \ref{iid}, \ref{loclip} and \ref{assum:dissipativity} hold. Then, there exist constants $\dot{c}, C_4, C_5, C_6>0$ such that, for every $\beta>0$, $0<\lambda\leq \lambda_{\max}$, and $n\in\mathbb{N},$
\begin{equation*}
W_2(\mathcal{L}(\theta^{\lambda}_n),\pi_{\beta})\leq C_4 \rme^{-\dot{c}\lambda n/4}(\mathbb{E}[|\theta_0|^4]+1) +(C_5+C_6)\lambda^{1/4},
\end{equation*}
where $\dot{c}$ is given in \eqref{dotc}, $C_4, C_5, C_6$ are given explicitly in \eqref{cw2const}. Moreover, for any $\varepsilon>0$, if we choose $\lambda \leq \frac{\varepsilon^4}{16(C_5+C_6)^4} \wedge \lambda_{\max}$, and 
\[
n\geq \frac{C_{\star}e^{C_{\star}(1+d/\beta)(1+\beta)}}{\varepsilon^4\dot{c}}\left(1+\frac{1}{(1-e^{-\dot{c}/2})^4}\right)\ln\left( \frac{C_{\star}e^{C_{\star}(1+d/\beta)(1+\beta)}}{\varepsilon}\left(1+\frac{1}{1-e^{-\dot{c}/2}}\right)\right)
\]
with $C_{\star}>0$ independent of $d, \beta, n$, then $W_2(\mathcal{L}(\theta^{\lambda}_n),\pi_{\beta})\leq \varepsilon$.
\end{corollary}


\begin{remark}
By \eqref{mainthmconst} and \eqref{cw2const}, one notes that the constants $C_2$ and $C_5$ are of order $\sqrt{d/\beta}$, and this implies that large values of $d$ on the upper bound can be controlled by large values of $\beta$. However, the constants $C_1$, $C_3$, $C_4$, $C_6$ have exponential dependence on $d$ and $\beta$, rather than $d/\beta$, due to the contraction result in \cite[Theorem 2.2]{eberle}. Furthermore, one notes that the undesirable dependence on $d$, which is derived under a geometric drift condition (and which, in turn, is implied by a dissipativity condition such as Assumption \ref{assum:dissipativity}), can be found, typically, in extreme cases, i.e. pathological examples of theoretical nature. This appears not to be the case in many practical applications. 
\end{remark}

\begin{remark}
In the case where $H(\theta, x) = h(\theta)$ for all $\theta \in \mathbb{R}^d$ and $x\in\mathbb{R}^m$, i.e. when the stochastic gradient coincides with the full gradient, Theorem \ref{main} and Corollary \ref{cw2} provide the full  non-asymptotic convergence results of the unadjusted Langevin algorithm (ULA) under dissipativity and Lipschitz continuity assumptions.
\end{remark}

Let $\hat{\theta} :=  \theta^{\lambda}_n$, where $\theta^{\lambda}_n$ denotes the $n$-th iteration of the SGLD algorithm \eqref{eq:discreteTimeSGLD}. Then, an upper bound for the expected excess risk $\mathbb{E}[U( \theta^{\lambda}_n)] - \inf_{\theta \in \mathbb{R}^d} U(\theta) $ can be obtained by using the following splitting: $
\mathbb{E}[U( \theta^{\lambda}_n)] - \inf_{\theta \in \mathbb{R}^d} U(\theta) = \left( \mathbb{E}[U( \theta^{\lambda}_n)]  -  \mathbb{E}[U(Z_{\infty})]\right) + \left( \mathbb{E}[U(Z_{\infty})]- \inf_{\theta \in \mathbb{R}^d} U(\theta) \right)$, where $Z_{\infty}\sim \pi_{\beta}$ with $\pi_{\beta}$ defined in \eqref{pibetaexp}. By using Corollary \ref{cw2}, an upper bound for the first term on the RHS of the above equality can be obtained as explained in \cite[Lemma~3.5]{raginsky}. The second term on the RHS of the above equality can be upper bounded by applying \cite[Proposition~3.4]{raginsky}. The precise statement with explicit constants is provided below.
\begin{corollary}\label{eer} Let Assumptions \ref{iid}, \ref{loclip} and \ref{assum:dissipativity} hold. Then, there exist constants $\dot{c}, C^{\sharp}_1, C^{\sharp}_2, C^{\sharp}_3>0$ such that, for every $\beta>0$, $0<\lambda\leq \lambda_{\max}$, $n\in\mathbb{N},$
\[
\mathbb{E}[U( \theta^{\lambda}_n)]  - \inf_{\theta \in \mathbb{R}^d} U(\theta)  \leq  C^{\sharp}_1e^{-\dot{c}\lambda n/4}+C^{\sharp}_2\lambda^{1/4}+C^{\sharp}_3,
\]
where $\dot{c}$ is given in \eqref{dotc}, $C^{\sharp}_1, C^{\sharp}_2, C^{\sharp}_3$ are given explicitly in \eqref{eerconst} and \eqref{eerconst2}. Moreover, for any $\varepsilon>0$, if we choose $\beta \geq \beta_{\varepsilon} \vee \frac{3d}{\varepsilon}\log\left(\frac{e L_1\mathbb{E}[\eta(X_0)]}{ad}\left(b+1\right)\left(d+1\right)\right) $ with $\beta_{\varepsilon}$ denoting the root of the function $f^\sharp(\beta) = \frac{\log\left( \beta+1\right)}{ \beta}-\frac{\varepsilon}{3d}$, $\lambda \leq \frac{\varepsilon^4}{81(C^{\sharp}_2)^4} \wedge \lambda_{\max}$, and $n\geq \frac{C_{\star}e^{C_{\star}(1+d/\beta)(1+\beta)}}{\varepsilon^4\dot{c}}\left(1+\frac{1}{(1-e^{-\dot{c}/2})^4}\right)\ln\left( \frac{C_{\star}e^{C_{\star}(1+d/\beta)(1+\beta)}}{\varepsilon}\left(1+\frac{1}{1-e^{-\dot{c}/2}}\right)\right)$ with $C_{\star}>0$ independent of $d, \beta, n$, then $\mathbb{E}[U( \theta^{\lambda}_n)]  - \inf_{\theta \in \mathbb{R}^d} U(\theta) \leq \varepsilon$.
\end{corollary}
\begin{remark} By \eqref{eerconst2}, one observes that $C^{\sharp}_3$ vanishes as $\beta$ tends to infinity. This implies that $\pi_{\beta}$ converges to a distribution which concentrates on the minimizers of $U$ for large enough $\beta$. This result provides a nonasymptotic bound for this concentration -- thus, sampling from $\pi_{\beta}$ solves the optimization problem of minimizing $U$. However, for a large $\beta$, it would require a large number of iterations for the SGLD algorithm to reach a given precision level measured using expected excess risk, see the expression for the lower bound of $n$ in Corollary \ref{eer}. This is due to a slow convergence of the Langevin dynamics \eqref{sdeintro} to the target distribution $\pi_{\beta}$ as the contraction constant $\dot{c}$ is inversely related to $\beta$, see \eqref{dotc}, which, in turn, lead to a slow convergence of SGLD to $\pi_{\beta}$. There is thus a trade-off between the precision of the approximation and the efficiency of the algorithm. Consequently, one may set $\beta = \beta_{\varepsilon} \vee \frac{3d}{\varepsilon}\log\left(\frac{e L_1\mathbb{E}[\eta(X_0)]}{ad}\left(b+1\right)\left(d+1\right)\right) $ so as to achieve a given precision level for the expected excess risk while ensuring that the SGLD algorithm takes the smallest possible number of iterations to sample approximately from $\pi_{\beta}$ (also within the given precision level). Furthermore, for any precision level, once the value for $\beta$ is specified, one may calculate the upper bound of $\lambda$ using the expression given in Corollary \ref{eer} and then set the upper bound as the value of $\lambda$ for the efficiency of the algorithm as $\lambda$ is negatively related to $n$.
\end{remark}
The proofs of Theorem \ref{main}, Corollary \ref{cw2}, \ref{eer} are postponed to Section~\ref{proofoverview}. Furthermore, the explicit expressions for the constants in the main results are summarised in Table \ref{tab:constantsexp} and \ref{tab:constantskeydep}.
\subsection{Related work and discussions}
\begin{table}[h]
\tiny
\begin{center}
\begin{minipage}{\textwidth}
\caption{Comparison of Theorem \ref{main} and Corollary \ref{cw2} with \cite[Proposition~3.3]{raginsky}, \cite[Theorem~1.4]{alex} and \cite[Theorem~2.5]{nonconvex}.}\label{sample-table}
\begin{tabular}{m{2em}>{\centering\arraybackslash\hspace{0pt}}m{7em}>{\centering\arraybackslash\hspace{0pt}}m{8em} >{\centering\arraybackslash\hspace{0pt}}m{7em} >{\centering\arraybackslash\hspace{0pt}}m{2em} >{\centering\arraybackslash\hspace{0pt}}m{12em}}
\toprule
 		 & Smoothness & Contractivity &$\textnormal{Var}(H(\theta, X_0))$& Data & Results\\
\midrule
  \cite{raginsky}& Globally Lipschitz $H$ in $\theta$ uniformly in $x$ &Uniform in $x$ dissipativity of $H$ & Bounded by $C|\theta|^2$& i.i.d.    & \multicolumn{1}{c}{$W_2(\mathcal{L}(\theta^{\lambda}_n),\pi_{\beta}) \leq C\lambda^{5/4}n$}\\
  \midrule
\cite{alex}  &Globally Lipschitz $h$ &  Convexity at infinity of $h$ &  Bounded by $C|\theta|^2\lambda^{\alpha}$, $\alpha >0$& i.i.d.    &$W_1(\mathcal{L}(\theta^{\lambda}_n),\pi_{\beta})\leq  C\lambda^{\alpha/2}$, \newline$W_2(\mathcal{L}(\theta^{\lambda}_n),\pi_{\beta})\leq  C\lambda^{\alpha/4}$\\
\midrule
\cite{nonconvex} &  Globally Lipschitz $H$ in  $\theta$ and $ x$& Uniform in $x$ dissipativity of $H$ &--- & $L$-mixing &$W_1(\mathcal{L}(\theta^{\lambda}_n),\pi_{\beta})\leq  C\lambda^{1/2}$ \\ 
\midrule
This paper &  Locally Lipschitz $H$ in $\theta$ and $ x$& Local dissipativity of $H$ &--- &i.i.d.    &$W_1(\mathcal{L}(\theta^{\lambda}_n),\pi_{\beta})\leq  C\lambda^{1/2}$, \newline$W_2(\mathcal{L}(\theta^{\lambda}_n),\pi_{\beta})\leq  C\lambda^{1/4}$ \\ 
\bottomrule
\end{tabular}
\end{minipage}
\end{center}
\end{table}

Under the preceding Assumptions \ref{iid}, \ref{loclip} and \ref{assum:dissipativity}, a convergence result in $W_1$ distance with rate 1/2 is given in Theorem \ref{main}, while in Corollary \ref{cw2}, a convergence result in $W_2$ distance with rate 1/4 is provided. \cite[Theorem~3.10]{convex} provides a convergence result in $W_2$ distance under similar assumptions in the convex setting, i.e. with Assumption \ref{assum:dissipativity} replaced by a strong convexity requirement. Moreover, the analysis of Theorem \ref{main} follows a similar approach as in \cite{nonconvex}, while its framework is crucially extended by assuming local Lipschitz continuity of $H$ in Assumption \ref{loclip}, and non-uniform estimates with respect to the $x$ variable in Assumption \ref{assum:dissipativity}.

Next, we mainly focus on the comparison of our work with \cite{raginsky} and \cite{alex}. In \cite[Proposition~3.3]{raginsky}, a finite-time convergence result of the SGLD algorithm \eqref{eq:discreteTimeSGLD} in Wasserstein-2 distance is provided, and the rate of convergence is shown to be $\lambda^{5/4}n$ with $n$ the number of iterations. To obtain this result, a dissipativity condition \cite[Assumption~(A.3)]{raginsky} is proposed. In \cite[Assumption~(A.1)]{raginsky}, the quantities $f(0,\cdot)$ and $H(0,\cdot)$ are assumed to be bounded, where $U(\theta)=\E[f(\theta,X_0)]$ and $H(\cdot,\cdot)=\nabla_{\theta}f(\cdot,\cdot)$, $\theta \in \mathbb{R}^d$. In addition, it requires the finiteness of an exponential moment of the initial value \cite[Assumption~(A.5)]{raginsky} and the Lipschitz continuity of $H$ in $\theta$ \cite[Assumption~(A.2)]{raginsky}. 
While Corollary \ref{cw2} improves the convergence rate provided in \cite{raginsky} in the sense that our rate of convergence does not depend on the number of iterations, we further require a local Lipschitz continuity of $H(\theta, x)$ in $x$. However, compared to \cite[Assumption~(A.3)]{raginsky}, we allow the dissipativity condition without imposing the uniformity in $x$ in Assumption \ref{assum:dissipativity}, and we require only polynomial moments of the initial value $\theta_0$. Furthermore, in Assumption \ref{loclip}, we relax the (global) Lipschitz condition of $H$ in $\theta$ by allowing the Lipschitz constant to depend on $x$. We note that \cite[Assumption~(A.4)]{raginsky} can be obtained by using Assumptions~\ref{iid} and \ref{loclip}.

Further, we compare our results with those in \cite{alex}. Compared to \cite[Theorem~1.4]{alex} with $\alpha = 1$, Theorem \ref{main} achieves the same rate in $W_1$ without assuming that the variance of the stochastic gradient is controlled by the stepsize \cite[Assumption~1.3]{alex}. To obtain Theorem \ref{main}, we assume a Lipschitz continuity of $H$ in Assumption~\ref{loclip}, while \cite[Theorem~1.4]{alex} requires a Lipschitz continuity of $h$ \cite[Assumption~1.1]{alex}. This latter condition on $h$ is implied by our Assumption \ref{loclip} as indicated in Remark \ref{rem:BoundsOnH}. It is typical though, for many real applications, that the full gradient $h$ is unknown, and crucially one can check conditions only for $H$ (as in Assumption \ref{loclip}) and not for $h$. This is also apparent in Section \ref{examplemain} where Assumptions \ref{loclip} and \ref{assum:dissipativity} are easily checkable for various examples. The same cannot be said for the corresponding conditions  regarding $h$. In a recent update of \cite{alex}, it is noted in \cite[Section~5.3]{alex} that the requirements for \cite[Theorem~1.4]{alex} can be potentially relaxed by replacing the convexity at infinity condition with a uniform dissipativity condition. This observation coincides with the results obtained in  \cite[Theorem~2.5]{nonconvex}, when one considers i.i.d. data, while we further generalise this framework by requiring only a local dissipativity condition, i.e. Assumption \ref{assum:dissipativity}.

\section{Applications}\label{examplemain}
In this section, we use $x^{\mathsf{T}}y$ for any $x, y \in \mathbb{R}^d$ for the inner product (instead of $\langle x, y \rangle$) to make the notation compact. 

\subsection{Variational inference for Bayesian logistic regression}\label{sec:VI}
Variational inference (VI) aims at approximating a posterior distribution $p(w | x)$, where $w$ is the quantity of interest and $x$ is the data, using a parameterized family of distributions which is often called the \textit{variational family} and is denoted by $q_\theta(w)$ \cite{wainwright2008graphical}. Thus, the VI approach converts an inference problem into an optimization problem, where the objective function is typically nonconvex. This enables us to apply our results to the VI problem and come up with theoretical guarantees. We present one such example below and provide guarantees for the application of the VI approach by using the conclusions of our main theorem.

Consider a probabilistic model consisting of a likelihood $p(w|x)$ and a prior $p(w)$. Here we implicitly assume the existence of probability density functions which are used to identify the corresponding probability distributions. To ease the notation, for a distribution $p$, its marginal, conditional and joint distributions are also denoted by $p$, however dependencies on appropriate state variables are explicitly declared. Furthermore, recall that given a joint distribution $p(w, x)$, one observes that for any distribution $q(w)$, $x \in \mathbb{R}^{\bar{m}}$, $\bar{m} \geq 1$, the following holds
\begin{align}\label{elbo}
\log p(x) 	&=\int_{\mathbb{R}^d}q(w)\log\left(\frac{p(w, x)}{q(w)}\right)\, dw+\int_{\mathbb{R}^d}q(w)\log\left(\frac{q(w)}{p(w| x)}\right)\, dw\nonumber\\
						& =  \mathbb{E}_{\mathbf{w}\sim q}\log \frac{p(\mathbf{w},x)}{q(\mathbf{w})}+ \KL(q(w)\|p(w|x)),
\end{align}
where the first term in \eqref{elbo} is usually denoted in VI literature by $\text{ELBO}(q) $. The aim is to choose a suitable approximating family $q_{\theta}$ parameterized by $\theta$, so as to minimize the KL divergence of the two distributions $q_{\theta}(w)$ and $p(w|x)$, for a given $x$, over ${\theta}$. This turns out to be equivalent to maximizing ELBO($q_{\theta}$) since $\log p(x)$ is fixed. One can decompose ELBO($q_{\theta}$)  = $l_1(\theta)+l_2(\theta)$ where $l_1(\theta) = \mathbb{E}_{\mathbf{w}\sim q_{\theta}}[\log p(\mathbf{w}, x)]$ and $l_2(\theta)$ is the entropy of $q_{\theta}$. Moreover, we suppose there exists a transformation $\mathcal{T}_{\theta}$ such that $\mathcal{T}_{\theta}(\mathbf{u}) \overset{\text{d}}{=} \mathbf{w}$. As a result, one obtains 
$
l_1(\theta) := \mathbb{E}_{\mathbf{u}\sim s}[ \log p(\mathcal{T}_{\theta}(\mathbf{u}), x) ] ,
$
and similarly, 
$
l_2(\theta)  := -\mathbb{E}_{\mathbf{u}\sim s}[\log q(\mathcal{T}_{\theta}(\mathbf{u}))  ].
$
This is called {\it reparameterization trick} in VI literature \cite{price1958useful, salimans2013fixed, kingma2013auto, rezende2014stochastic}. By using this technique, one can obtain stochastic estimates of $\nabla_{\theta} (l_1(\theta)+l_2(\theta))$ and then use SGLD algorithm to maximize ELBO($q_{\theta}$). 

We consider an example from Bayesian logistic regression \cite{wainwright2008graphical}. Suppose a collection of data points $\mathcal{X}   = \{(z_i, y_i)\}_{i = 1, \dots, n}$ is given, where $z_i \in \mathbb{R}^d$ and $y_i \in \{0,1\}$ for all $i$. Denote by $\mathcal{Z}_i  = (z_i, y_i)$ for all $i$, then $\mathcal{X} = \{\mathcal{Z}_i \}_{i = 1, \dots, n}$. Assume Gaussian mixture prior to define a multimodal distribution characterized by $p(w, \mathcal{X})  = \pi_0(w)\prod_{i = 1}^n p(\mathcal{Z}_i|w)$, where $\pi_0(w)$ is the prior given by
$
\pi_0( w) \wasypropto  \exp(-\bar{f}(w)) = e^{-|w - \hat{a}|^2/2}+e^{-|w+\hat{a}|^2/2}
$
with $\bar{f}(w) = |w - \hat{a}|^2/2 -\log(1+\exp(-2\hat{a}^{\mathsf{T}}w ))$, $\hat{a}\in \mathbb{R}^d$, $|\hat{a}|^2>1$ and $p(\mathcal{Z}_i|w) = (1/(1+e^{-z_i^{\mathsf{T}}w}))^{y_i}(1-1/(1+e^{-z_i^{\mathsf{T}}w}))^{1-y_i}$ is the likelihood function.
Moreover, take a variational distribution parameterized by $\theta$, which is given as
$
 q_\theta( w) \wasypropto e^{-|w - \theta|^2/2}+e^{-|w+\theta|^2/2}.
$
Then, maximizing $l_1(\theta)+l_2(\theta) = \mathbb{E}_{\mathbf{w}\sim q_{\theta}}[\log p(\mathbf{w}, \mathcal{X}) - \log q_\theta( \mathbf{w}) ]$ in $\theta$ is equivalent to maximizing the following:
\begin{align}\label{vieq1}
& \mathbb{E}_{\mathbf{w}\sim q_{\theta}}\left[-|\mathbf{w} - \hat{a}|^2/2 +\log(1+\exp(-2\hat{a}^{\mathsf{T}}\mathbf{w} )) \right.\nonumber\\
&\hspace{3em}\left. +\sum_{i = 1}^n(-y_i \log(1+e^{-z_i^{\mathsf{T}}\mathbf{w}})+(y_i-1)\log (1+e^{z_i^{\mathsf{T}}\mathbf{w}}))\right] \nonumber \\
& +\mathbb{E}_{\mathbf{w}\sim q_{\theta}}\left[|\mathbf{w} - \theta|^2/2 -\log(1+\exp(-2\theta^{\mathsf{T}}\mathbf{w} ))\right].
\end{align}
Further, the reparameterization technique is applied by considering the mapping $\mathcal{T}_{\theta}(u):= \mathbbm{1}_{\{v = 1\}}(Cu+m)+\mathbbm{1}_{\{v = 0\}}(Cu-m)$ where $v \sim \text{Ber}(q)$ and $\theta = (C, m, q)$. Here, we fix $C = \mathbf{I}_d/4$, $q = 7/8$, and thus $\mathcal{T}_{\theta}(u) = \mathbbm{1}_{\{v = 1\}}(u/4+\theta)+\mathbbm{1}_{\{v = 0\}}(u/4-\theta)$. 
Then, for $\mathbf{u}\sim s$ where $s$ is the standard Gaussian distribution, the expression in \eqref{vieq1} becomes
\begin{align}\label{expressionUVI}
\begin{split}
l_1(\theta)+l_2(\theta) &= \frac{7}{8}\mathbb{E}_{\mathbf{u}\sim s}\left[-|\mathbf{u}/4+\theta- \hat{a}|^2/2 +\log(1+\exp(-2\hat{a}^{\mathsf{T}}(\mathbf{u}/4+\theta) )) \right.  \\
&\quad \left. +\sum_{i = 1}^n(-y_i \log(1+e^{-z_i^{\mathsf{T}}(\mathbf{u}/4+\theta)})+(y_i-1)\log (1+e^{z_i^{\mathsf{T}}(\mathbf{u}/4+\theta)}))\right]  \\
&\quad +\frac{1}{8}\mathbb{E}_{\mathbf{u}\sim s}\left[-|\mathbf{u}/4-\theta- \hat{a}|^2/2 +\log(1+\exp(-2\hat{a}^{\mathsf{T}}(\mathbf{u}/4-\theta) )) \right.  \\
&\quad \left. +\sum_{i = 1}^n(-y_i \log(1+e^{-z_i^{\mathsf{T}}(\mathbf{u}/4-\theta)})+(y_i-1)\log (1+e^{z_i^{\mathsf{T}}(\mathbf{u}/4-\theta)}))\right] \\
&\quad+\frac{7}{8}\mathbb{E}_{\mathbf{u}\sim s}\left[|\mathbf{u}/4+\theta- \theta|^2/2 -\log(1+\exp(-2(\theta^{\mathsf{T}}\mathbf{u}/4+|\theta|^2) )) \right] \\
&\quad+\frac{1}{8}\mathbb{E}_{\mathbf{u}\sim s}\left[|\mathbf{u}/4-\theta- \theta|^2/2 -\log(1+\exp(-2(\theta^{\mathsf{T}}\mathbf{u}/4-|\theta|^2) )) \right].
\end{split}
\end{align}
In what follows, we derive the stochastic gradient expression for the cost function defined in Eq.~\eqref{expressionUVI}. We note that, stochasticity in this example comes from sampling $\mathbf{u}$ variables and constructing empirical expectation estimates, rather than subsampling data points. 
\begin{proposition}\label{prop:VI} Let the objective function of the VI example be defined in \eqref{expressionUVI}. Moreover, let $\mathbf{u}$ be a standard $d$-dimensional Gaussian random variable, and let $(\mathbf{u}_n)_{n \in \mathbb{N}}$ be a sequence of i.i.d. standard $d$-dimensional Gaussian random variables. In addition, assume $|\theta_0| \in L^4$.  Let  $H: \mathbb{R}^d \times \mathbb{R}^d \rightarrow \mathbb{R}^d$ be the stochastic gradient of \eqref{expressionUVI} given by 
\begin{align}\label{HexpressionVI}
\begin{split}
H(\theta, u) & = \frac{\theta}{2} + \frac{u}{4}-\frac{3}{4}\hat{a}+ \frac{\hat{a}}{4}\left(\frac{7}{1+e^{2\hat{a}^{\mathsf{T}}(u/4+\theta)}}-\frac{1}{1+e^{2\hat{a}^{\mathsf{T}}(u/4-\theta)}}\right)\\
&\quad+ \frac{1}{8}\sum_{i = 1}^n\left( -6z_i y_i +\frac{7z_i}{1+e^{-z_i^{\mathsf{T}}(u/4+\theta)}}-\frac{ z_i}{1+e^{-z_i^{\mathsf{T}}(u/4-\theta)}}\right)\\
&\quad -\frac{7(u+8\theta)}{16(1+e^{2(\theta^{\mathsf{T}}u/4+|\theta|^2)})}-\frac{u-8\theta}{16(1+e^{2(\theta^{\mathsf{T}}u/4-|\theta|^2)})}.
\end{split}
\end{align} 
Then, $H$ satisfies Assumptions \ref{iid}, \ref{loclip}, and \ref{assum:dissipativity}. More precisely, Assumption \ref{loclip} holds with $L_1 = 1$, $L_2 = 1/4$, $\eta(u)  = 9/2+8e^{|u|^2/32}+\sum_{i = 1}^n|z_i|^2+4|\hat{a}|^2+3|u|^2/8$, and moreover, $\E[(1+|\mathbf{u}|)^4\eta^4(\mathbf{u})]<\infty$.
Assumption \ref{assum:dissipativity} holds with $A(u) = \mathbf{I}_d /4$ and $\hat{b}(u)  =   (9| u|^2/4+121|\hat{a}|^2/4)+49n\sum_{i = 1}^n|z_i|^2/8+7n/4$. 
\end{proposition}
The proof of Proposition \ref{prop:VI} is postponed to Appendix \ref{proof:examplemain}. Moreover, the meaning of this result is that we can verify, in a practical variational inference example, that our local assumptions are satisfied, therefore in this instance of nonconvex optimization problem, the theoretical guarantees we provide in this paper regarding SGLD provably hold. This is a significant improvement over previous results which are based on global Lipschitz assumptions, which fail to hold for this simple, yet very illustrative, example.

\subsection{Index tracking}\label{sec:IndTra}
We consider the problem of index tracking, which can be formulated precisely as follows (see, e.g., \cite[Eqn. (3), (4), and (8)]{gaivoronski2005optimal}):
\begin{equation}\label{eq:itob}
\min_{\theta} U(\theta)  := \min_{\theta} \left(\E\left[\left(Y -  \sum_{i = 1}^N g_i(\theta) X_{i}\right)^2\right]+\hat{\eta}|\theta|^2\right),
\end{equation}
where $\theta \in \R^N$, $U: \R^N \rightarrow \R$, and $Z= (Y,X_1, \dots, X_N)$ is an $\R^{N+1}$-valued random variable with $Y \in \R$ denoting the return of the target index, $X_i \in \R, i = 1, \dots, N$ denoting the return of the $i$-th asset. Moreover, $\hat{\eta} >0$ is the regularization constant, and $g_i(\theta)$ represents the weight of asset $i$ in the portfolio given explicitly by $g_i(\theta) = \frac{e^{\theta_i}}{\sum_{k=1}^N e^{\theta_k}}$, for all $\theta \in \R^N$. One notes that for any $\theta \in \R^N$, $i = 1, \dots, N$, $g_i(\theta) \in (0,1)$. Moreover, for any $\theta \in \R^N$, $m = 1, \dots, N$, one obtains
\begin{align}\label{eq:itobgrad}
\partial_{\theta_m}U(\theta)  = 2\hat{\eta} \theta_m+2\E\left[ \left(Y -  \sum_{i = 1}^N g_i(\theta) X_{i}\right) g_m(\theta)\sum_{i \neq m}^N g_i(\theta)(X_i- X_m)\right].
\end{align}
\begin{proposition}\label{prop:IndTrasg} Let the objective function $U$ be defined in \eqref{eq:itob}. Denote by $\mathcal{L}(Z)$ the probability law of $Z$. Let $(Z_n)_{n \in \mathbb{N}}$ be a sequence of i.i.d. random variables with probability law $\mathcal{L}(Z)$. Assume $|\theta_0| \in L^4$, and $|Z| \in L^{12}$. Moreover, let $H: \R^N \times \R^{N+1} \rightarrow \R^N$ be the stochastic gradient of \eqref{eq:itob} given by
\begin{equation}\label{eqn:itsg}
H(\theta, z) := (H_1(\theta,z), \dots, H_N(\theta,z)),
\end{equation}
where $z := (y, x_1, \dots, x_N) \in \R^{N+1}$, $H_m:  \R^N  \times \R^{N+1} \rightarrow \R $, $m = 1, \dots, N$. The explicit expressions for $H_m$, $m = 1, \dots, N$ are given as follows:
\begin{align}\label{eqn:itsgexp}
H_m(\theta,z) 
& = 2\hat{\eta} \theta_m +2\left(y -  \sum_{i = 1}^N g_i(\theta) x_{i}\right) g_m(\theta)\sum_{i \neq m}^N g_i(\theta)(x_i- x_m).
\end{align}
Then, the following holds:
\begin{enumerate}
\item The function $U$ is in general nonconvex. 
\item The stochastic gradient $H$ satisfies Assumptions \ref{iid}, \ref{loclip}, and \ref{assum:dissipativity}. More precisely, Assumption \ref{loclip} holds with $L_1  =6N $, $L_2  =   4\sqrt{N}(N+1)$, $\eta(z)  =\hat{\eta}+ \left(1+|y|+\sum_{i = 1}^N|x_i|\right)\left(1 +\sum_{i \neq m}^N(|x_i|+|x_m|)\right)$. Assumption \ref{assum:dissipativity} holds with $A(z)  = \hat{\eta}\mathbf{I}_N $ and $\hat{b}(z)  =  \hat{\eta}^{-1}N\left( \left(|y|+\sum_{i = 1}^N|x_i|\right)\sum_{i \neq m}^N(|x_i|+|x_m|)\right)^2$. 
\end{enumerate}
\end{proposition}
The proof of Proposition \ref{prop:IndTrasg} is postponed to Appendix \ref{proof:examplemain}. One notes that  the stochastic gradient  $H(\theta, z)$ fails to satisfy the global conditions as the index tracking optimization \eqref{eq:itob} can be viewed as an online optimization problem where a uniform bound of the data $z$ is unavailable. However, it is shown in Proposition \ref{prop:IndTrasg} that, for any $\theta \in \R^N, z \in \R^{N+1}$, $H(\theta,z)$ satisfies the local Lipschitz condition (Assumption \ref{loclip}) and the local dissipativity condition (Assumption \ref{assum:dissipativity}). Moreover, Proposition \ref{prop:IndTrasg} implies that our main results hold for the nonconvex optimization problem \eqref{eq:itob}, which provide theoretical guarantees for the SGLD algorithm to find the approximate minimizers.

\section{Proof Overview}\label{proofoverview}
In this section, we explain the main idea of proving Theorem \ref{main} and Corollary \ref{cw2}. We proceed by introducing suitable Lyapunov functions for the analysis of moment estimates of the SGLD algorithm \eqref{eq:discreteTimeSGLD}. This is done with the help of a continuous-time interpolation of the original recursion \eqref{eq:discreteTimeSGLD}, which results in a continuous-time process whose laws at discrete times are the same as those of the SGLD. We further introduce a couple of auxiliary continuous-time process, which are used for the derivation of preliminary results. The proof of the main results then follows. We defer all proofs to Appendix \ref{proof:sec24} and focus on the exposition of main ideas.
\subsection{Introduction of suitable Lyapunov functions and auxiliary processes}
We start by defining, for each $p\geq 1$, the Lyapunov function $V_p$ by $V_p(\theta):=(1+|\theta|^2)^{p/2},\ \theta\in\mathbb{R}^d$, and similarly $\operatorname{v}_p (\omega):= (1+\omega^2)^{p/2}$, for any real $\omega\ge0$. Notice that these functions are twice continuously differentiable and 
\begin{equation}\label{contrassumption}
\sup_{\theta} (|\nabla V_p(\theta)|/ V_p(\theta)) <\infty, \quad \lim_{|\theta|\to\infty} (\nabla V_p(\theta)/V_p(\theta))=0.
\end{equation} 
Let $\mathcal{P}_{\, V_p}$ denote the set of $\mu\in\mathcal{P}(\mathbb{R}^d)$ satisfying $\int_{\mathbb{R}^d}V_p(\theta)\,\mu(\rmd \theta)<\infty$. 

Consider the Langevin SDE $(Z_t)_{t\in\bR_+}$ given by
\begin{equation}\label{sde}
\rmd Z_t := -  h(Z_t) \rmd t + \sqrt{2 \beta^{-1}} \rmd B_t
\end{equation}
with $Z_0:=\theta_0 \in \R^d $, where $h: = \nabla U$ and $(B_t)_{t \ge 0}$ is a standard $d$-dimensional Brownian motion. Denote by  $(\mathcal{F}_t)_{t \geq 0}$ the natural filtration of $(B_t)_{t \geq 0}$, and we assume that $(\mathcal{F}_t)_{t \geq 0}$ is independent of $\mathcal{G}_{\infty}\vee \sigma(\theta_0)$. Moreover, denote by $\mathcal{F}_{\infty} := \sigma(\bigcup_{t \geq 0} \mathcal{F}_t)$.

We next introduce the auxiliary processes which are used in our analysis. For each $\lambda>0$, $Z^{\lambda}_t:=Z_{\lambda t},\ t\in\mathbb{R}_+$, where the process $(Z_t)_{t\in\bR_+}$ is defined in \eqref{sde}. We also define $\tilde{B}^{\lambda}_t:=B_{\lambda t}/\sqrt{\lambda}$, $t \geq 0$. We note that $(\tilde{B}_t^\lambda)_{t \geq 0}$ is a Brownian motion and
\[
\rmd Z^{\lambda}_t:=-\lambda h(Z^{\lambda}_t)\, \rmd t+\sqrt{2\lambda\beta^{-1}} \rmd \tilde{B}^{\lambda}_t,\quad
Z^{\lambda}_0:=\theta_0.
\]
The natural filtration of $(\tilde{B}^{\lambda}_t)_{t \geq 0}$ is denoted by $(\mathcal{F}_t^{\lambda})_{t\geq 0}$ with $\mathcal{F}_t^{\lambda}:=\mathcal{F}_{\lambda t}$, $t\in\mathbb{R}_+$. Note that $(\mathcal{F}_t^{\lambda})_{t \geq 0}$ is independent of $\mathcal{G}_{\infty}\vee \sigma(\theta_0)$. 

Then, define the continuous-time interpolation of the SGLD algorithm \eqref{eq:discreteTimeSGLD} as
\begin{equation}\label{SGLDprocess}
\rmd \bar{\theta}^{\lambda}_t:=-\lambda H(\bar{\theta}^{\lambda}_{\lfloor t\rfloor},{X}_{\lceil t\rceil})\, \rmd t
+ \sqrt{2\lambda\beta^{-1}} \rmd \tilde{B}^{\lambda}_{t}, \quad \bar{\theta}^{\lambda}_0:=\theta_0.
\end{equation}
In addition, 
one notes the law of the interpolated process coincides with the law of the SGLD algorithm \eqref{eq:discreteTimeSGLD} at grid-points, i.e. $\mathcal{L}(\bar{\theta}^{\lambda}_n):=\mathcal{L}(\theta_n^{\lambda})$, for each $n\in\mathbb{N}$. Hence, crucial estimates for the SGLD can be derived by studying  equation \eqref{SGLDprocess}. 

Furthermore, consider a continuous-time process $ \zeta^{s,v, \lambda}_t$, $t\geq s$, which denotes the solution of the SDE
\begin{equation*}
\rmd \zeta^{s,v, \lambda}_t:= -\lambda h(\zeta^{s,v, \lambda}_t) \rmd t + \sqrt{2\lambda\beta^{-1}} \rmd \tilde{B}_t^{\lambda}, \quad \zeta^{s,v, \lambda}_s := v \in \mathbb{R}^d.
\end{equation*}
\begin{definition}\label{zetaprocess} Fix $n \in \nset$. For any $t \geq nT$, define $\bar{\zeta}_t^{\lambda,n} := \zeta^{nT, \bar{\theta}^{\lambda}_{nT}, \lambda}_t$, where $T := \floor{{1}/{\lambda}}$.
\end{definition}
Intuitively, $\bar{\zeta}_t^{\lambda,n}$ is a process started from the value of the SGLD process \eqref{SGLDprocess} at time $nT$ and run until time $t \geq nT$ with the continuous-time Langevin dynamics.

\subsection{Preliminary estimates}\label{mesec}
It is a classic result that SDE \eqref{sde} has a unique solution adapted to $(\mathcal{F}_t)_{t\in\mathbb{R}_+}$, since $h$ is Lipschitz-continuous by \eqref{mulyan}. Note that the second moments of the SDE \eqref{sde}  and of the target distribtuion $\pi_{\beta}$ are finite, see \cite[Lemma~3.2]{raginsky} and \cite[Proposition 1-(ii)]{aew}. For results of SDEs under local Lipschitz conditions, see, e.g. \cite{mshloc} and \cite{chjloc}. 

In order to obtain the convergence results, we first establish the moment bounds of the process $(\bar{\theta}^{\lambda}_t)_{t\geq 0}$. This result proves that second and fourth moments of the process $(\bar{\theta}^{\lambda}_t)_{t\geq 0}$ are uniformly bounded, therefore well-behaved.

\begin{lemma}
\label{lem:moment_SGLD_2p}
Let Assumptions \ref{iid}, \ref{loclip} and \ref{assum:dissipativity} hold. For any $0<  \lambda < \lambda_{\max}$ given in \eqref{eq:definition-lambda-max}, $n\in\mathbb{N}$, $t \in (n,n+1]$,
\[
\mathbb{E} \left[|\bar{\theta}^{\lambda}_t|^2\right] \leq
(1-  a\lambda (t-n) )(1-a\lambda)^n \mathbb{E}\left[ |\theta_0|^2\right] +c_1 (\lambda_{\max}+a^{-1}) \, ,
\]
where
\begin{align}\label{2ndmomentconst}
\begin{split}
c_1 &:= c_0+ 2d/\beta, \quad c_0 := 4\lambda_{\max}L_2^2\mathbb{E}\left[\bar{\eta}^2(X_0)\right] +4\lambda_{\max}H_\star^2 +2 b.
\end{split}
\end{align}
In addition, $\sup_t\mathbb{E} \left[|\bar{\theta}^{\lambda}_t|^2 \right]\leq \mathbb{E} \left[|\theta_0|^2\right] +c_1 (\lambda_{\max}+a^{-1})< \infty.$ Similarly, one obtains
\[
\mathbb{E}\left[|\bar{\theta}^{\lambda}_t|^4\right] \leq (1-  a\lambda (t-n) )(1-a\lambda)^n\mathbb{E}\left[| \theta_0|^4\right] +c_3 (\lambda_{\max}+a^{-1}),
\]
where
\begin{align}\label{4thmomentconst}
\begin{split}
c_3 	&:=  (1+a\lambda_{\max})c_2+12d^2\beta^{-2}(\lambda_{\max}+9a^{-1}),\\
c_2 	&:= 4bM^2 + 152(1+\lambda_{\max})^3 \\
&\quad \times\left((1+L_2)^4\mathbb{E}\left[(1+\bar{\eta}(X_0))^4\right]+(1+H_\star)^4\right)(1+M)^2,\\
M 		&:= \max\{(8ba^{-1}+48a^{-1}\lambda_{\max}(L_2^2\mathbb{E}\left[\bar{\eta}^2(X_0)\right]+H_\star^2))^{1/2},\\
		&\hspace{4em} (128a^{-1}\lambda_{\max}^2(L_2^3\mathbb{E}\left[\bar{\eta}^3(X_0)\right]+H_\star^3))^{1/3}\}.
\end{split}
\end{align}
Moreover, this implies $\sup_t\mathbb{E} |\bar{\theta}^{\lambda}_t|^4 < \infty$.
\end{lemma}
The uniform bound achieved in Lemma~\ref{lem:moment_SGLD_2p} for the fourth moment of the process $(\bar{\theta}_t^\lambda)_{t\geq 0}$ enables us to obtain a uniform bound for $V_4(\theta_t^{\lambda})$ as presented in the following corollary.

\begin{corollary}\label{corr:Moments} Let Assumptions \ref{iid}, \ref{loclip} and \ref{assum:dissipativity} hold. For any $0< \lambda < \lambda_{\max}$ given in \eqref{eq:definition-lambda-max}, $n\in\mathbb{N}$, $t \in (n,n+1]$,
\begin{align*}
\E[V_4(\bar{\theta}^{\lambda}_t)] \leq 2(1 - a\lambda)^\floor{t} \E[V_4(\theta_0)]  + 2  c_3  (\lambda_{\textnormal{max}} + a^{-1})+2,
\end{align*}
where $c_3$ is given in \eqref{4thmomentconst}.
\end{corollary}

Next, we turn our attention to the process $(\bar{\zeta}^{\lambda,n}_t)_{t\in\bR}$. 
We first present a drift condition associated with the SDE \eqref{sde}, which will be used to obtain the moment bounds of the process $\bar{\zeta}^{\lambda,n}_t$.
\begin{lemma}[{\cite[Lemma~3.6]{nonconvex}}]\label{lem:PreLimforDriftY}
Let Assumptions~\ref{iid} and \ref{assum:dissipativity} hold. Then, for each $p\geq 2$, $\theta \in \rset^d$, 
\begin{align*}
\Delta V_p(\theta)/\beta- \langle h(\theta), \nabla V_p(\theta) \rangle \leq -\bar{c}(p) V_p(\theta) + \tilde{c}(p), 
\end{align*}
where $\bar{c}(p):= ap/4$ and $ \tilde{c}(p) := (3/4) a p \mathrm{v}_p(\overline{M}_p)$ with $\overline{M}_p := (1/3 + 4b/(3a) + 4d/(3a\beta) + 4(p-2)/(3a\beta))^{1/2}.$
\end{lemma}

The following lemma provides explicit upper bounds for $V_p(\bar{\zeta}_t^{\lambda,n})$ in the case $p = 2$ and $p=4$.
\begin{lemma}\label{zetaprocessmoment} Let Assumptions \ref{iid}, \ref{loclip} and \ref{assum:dissipativity} hold.  For any $0< \lambda < \lambda_{\max}$ given in \eqref{eq:definition-lambda-max}, $t \geq nT$, $n\in \nset$, one obtains the following inequality
\begin{align*}
\E[V_2(\bar{\zeta}_t^{\lambda,n})] &\leq e^{-a\lambda  t/2}   \E[V_2(\theta_0)] +3 \mathrm{v}_2(\overline{M}_2)+c_1(\lambda_{\max}+a^{-1})+1,
\end{align*}
where the process $\bar{\zeta}_t^{\lambda,n}$ is defined in Definition \ref{zetaprocess} and $c_1$ is given in \eqref{2ndmomentconst}. Furthermore,
\begin{align*}
\E[V_4(\bar{\zeta}_t^{\lambda,n})] \leq 2e^{-a\lambda  t}   \E[V_4(\theta_0)] +3 \mathrm{v}_4(\overline{M}_4)+2c_3 (\lambda_{\max}+a^{-1})+2,
\end{align*}
where $c_3$ is given in \eqref{4thmomentconst}.
\end{lemma}

\subsection{Proof of the main theorems}\label{fop}

We are now ready to establish our main results. Recall that, our goal is to establish a non-asymptotic bound for $W_1(\mathcal{L}(\theta^{\lambda}_n),\pi_{\beta})$. 

We first split $W_1(\mathcal{L}(\theta^{\lambda}_n),\pi_{\beta})$ as follows by using triangle inequality: 
$W_1(\mathcal{L}(\theta^{\lambda}_n),\pi_{\beta}) \leq W_1(\mathcal{L}(\bar{\theta}^{\lambda}_n),\mathcal{L}(Z^{\lambda}_n))+W_1(\mathcal{L}(Z^{\lambda}_n),\pi_{\beta})$. 
We aim at bounding the two terms on the right hand side separately. To achieve this, we first introduce a functional which is crucial to obtain the convergence rate in $W_1$. For any $p\geq 1$,
$ \mu,\nu \in \mathcal{P}_{\, V_p}$,
\begin{equation}\label{w1p}
w_{1,p}(\mu,\nu):=\inf_{\zeta\in\mathcal{C}(\mu,\nu)}\int_{\mathbb{R}^d}\int_{\mathbb{R}^d} [1\wedge |\theta-\theta'|](1+V_p(\theta)+V_p(\theta'))\zeta(\rmd\theta \rmd\theta'),
\end{equation}
and it satisfies trivially $W_1(\mu,\nu)\leq w_{1,p}(\mu,\nu)$. The case $p=2$, i.e. $w_{1,2}$, is used throughout the section.  The result below states a contraction property of $w_{1,2}$.
\begin{proposition} \label{contr} Let Assumptions \ref{iid}, \ref{loclip} and \ref{assum:dissipativity} hold. Let $Z_t'$, $t\in\rset_+$ be the solution of \eqref{sde} with initial condition $Z'_0 = \theta'_0$ which is independent of $\calF_\infty$ and satisfies $|\theta_0'| \in L^2$. Then, 
\[
w_{1,2}(\calL(Z_t),\calL(Z'_t)) \leq \hat{c} e^{-\dot{c} t} w_{1,2}(\calL(\theta_0),\calL(\theta_0')),
\] 
where the constants $\dot{c}$ and $\hat{c}$ are given in Lemma \ref{contractionconst}.
\end{proposition}
\begin{proof}
One notes that \cite[Assumption 2.1]{eberle} holds with $\kappa = L_1\E[\eta(X_0)]$ due to Remark \ref{rem:BoundsOnH}. \cite[Assumption 2.2]{eberle} holds with $V = V_2$ due to Lemma \ref{lem:PreLimforDriftY}. Moreover, \cite[Assumptions 2.4 and 2.5]{eberle} hold due to \eqref{contrassumption}. Thus, \cite[Theorem 2.2, Corollary 2.3]{eberle} hold under Assumptions~\ref{iid}, \ref{loclip} and \ref{assum:dissipativity}. Then, the desired result can be obtained by using the same argument as in the proof of \cite[Proposition~3.14]{nonconvex}.
\end{proof}

Now recall $T:= \floor{{1}/{\lambda}}$ defined in Definition \ref{zetaprocess}. By using the contraction property provided in Proposition \ref{contr}, one can construct the non-asymptotic bound between $\mathcal{L}(\bar{\theta}^{\lambda}_t)$ and $\mathcal{L}(Z^{\lambda}_t)$ in $W_1$ distance by decomposing the error using the auxiliary process $\bar{\zeta}_t^{\lambda,n}$:
\begin{align}\label{decomposition}
\begin{split}
W_1(\mathcal{L}(\bar{\theta}^{\lambda}_t),\mathcal{L}(Z^{\lambda}_t)) &\leq W_1(\mathcal{L}(\bar{\theta}^{\lambda}_t),\mathcal{L}(\bar{\zeta}_t^{\lambda,n}))  + W_1(\mathcal{L}(\bar{\zeta}_t^{\lambda,n}),\mathcal{L}(Z^{\lambda}_t)).
\end{split}
\end{align}
By the definition of $\lambda_{\max}$ given in \eqref{eq:definition-lambda-max}, we have that $0<\lambda \leq \lambda_{\max} \leq 1$, which implies $1/2 < \lambda T \leq 1$. An upper bound for the first term in \eqref{decomposition} is obtained below.
\begin{lemma}\label{convergencepart1} Let Assumptions \ref{iid}, \ref{loclip} and \ref{assum:dissipativity} hold. For any $0< \lambda < \lambda_{\max}$ given in \eqref{eq:definition-lambda-max}, $t \in (nT, (n+1)T]$,
\begin{align*}
W_2(\mathcal{L}(\bar{\theta}^{\lambda}_t),\mathcal{L}(\bar{\zeta}_t^{\lambda,n}))
\leq   \sqrt{\lambda} (e^{-an/4}\bar{C}_{2,1}\mathbb{E}[V_2(\theta_0)]  +\bar{C}_{2,2})^{1/2},
\end{align*}
where $\bar{C}_{2,1}$ and $\bar{C}_{2,2}$ are given in \eqref{barc2}.
\end{lemma}


Then, the following Lemma provides the bound for the second term in \eqref{decomposition}.
\begin{lemma}\label{convergencepart2} Let Assumptions \ref{iid}, \ref{loclip} and \ref{assum:dissipativity} hold. For any $0< \lambda < \lambda_{\max}$ given in \eqref{eq:definition-lambda-max}, $t \in (nT, (n+1)T]$,
\begin{align*}
W_1(\mathcal{L}(\bar{\zeta}_t^{\lambda,n}),\calL(Z_t^\lambda)) \leq \sqrt{\lambda}(e^{-\dot{c}n/2}\bar{C}_{2,3}\mathbb{E}[V_4(\theta_0)] +\bar{C}_{2,4}),
\end{align*}
where $\bar{C}_{2,3} $, $\bar{C}_{2,4}$ are given in \eqref{mainthmc2}.
\end{lemma}
By using similar arguments as in Lemma \ref{convergencepart2}, one can obtain the non-asymptotic estimate in $W_2$ distance between $\mathcal{L}(\bar{\zeta}_t^{\lambda,n})$ and $\calL(Z_t^\lambda)$, which is given in the following corollary.
\begin{corollary}\label{convergencepart2w2} Let Assumptions \ref{iid}, \ref{loclip} and \ref{assum:dissipativity} hold. For any $0< \lambda < \lambda_{\max}$ given in \eqref{eq:definition-lambda-max}, $t \in (nT, (n+1)T]$,
\begin{align*}
W_2(\mathcal{L}(\bar{\zeta}_t^{\lambda,n}),\calL(Z_t^\lambda)) \leq \lambda^{1/4}(e^{-\dot{c}n/4}\bar{C}^*_{2,3}\mathbb{E}^{1/2}[V_4(\theta_0)] +\bar{C}^*_{2,4}),
\end{align*}
where $\bar{C}^*_{2,3} $, $\bar{C}^*_{2,4}$ are given in \eqref{mainthmcstar2}.
\end{corollary}

Finally, by using the inequality \eqref{decomposition} and the results from previous lemmas, one can obtain the non-asymptotic bound between $\mathcal{L}(\bar{\theta}^{\lambda}_t)$ and $\mathcal{L}(Z^{\lambda}_t)$, $t \in (nT, (n+1)T]$, in $W_1$ distance.
\begin{lemma}\label{convergencepart3} Let Assumptions \ref{iid}, \ref{loclip} and \ref{assum:dissipativity} hold. For any $0< \lambda < \lambda_{\max}$ given in \eqref{eq:definition-lambda-max}, $t \in (nT, (n+1)T]$,
\begin{align*}
W_1(\mathcal{L}(\bar{\theta}^{\lambda}_t),\mathcal{L}(Z^{\lambda}_t)) \leq (\bar{C}_{2,1}^{1/2} +\bar{C}_{2,2}^{1/2}+ \bar{C}_{2,3}+\bar{C}_{2,4})\sqrt{\lambda}(e^{-\dot{c}n/2}\mathbb{E}[V_4(\theta_0)] +1),
\end{align*}
where $\bar{C}_{2,1}$, $\bar{C}_{2,2}$ are given in \eqref{barc2} (Lemma \ref{convergencepart1}), and $\bar{C}_{2,3} $, $\bar{C}_{2,4}$ are given in \eqref{mainthmc2} (Lemma \ref{convergencepart2}).
\end{lemma}

Before proceeding to the proofs of the main results, the constants $\dot{c}$ and $\hat{c}$ from Proposition \ref{contr} are given in an explicit form.

\begin{lemma} \label{contractionconst} The contraction constant $\dot{c}>0$ in Proposition \ref{contr} is given by
\begin{equation} \label{dotc}
\dot{c}:=\min\{\bar{\phi}, \bar{c}(2), 4\tilde{c}(2) \epsilon\bar{c}(2)\}/2,
\end{equation}
where $\bar{c}(2)= a/2$, $ \tilde{c}(2) = (3/2) a  \mathrm{v}_2(\overline{M}_2)$ with $\overline{M}_2$ given in Lemma~\ref{lem:PreLimforDriftY}, $\bar{\phi}$ is given by
\begin{equation} \label{simplifiedphi}
\bar{\phi}:= \left(\bar{b}\sqrt{8\pi/(\beta K_1)}   \exp\left(\left(\bar{b} \sqrt{\beta K_1/8} +\sqrt{8/(\beta K_1)}\right)^2\right) \right)^{-1} \,,
\end{equation}
and moreover, $\epsilon>0$ can be chosen such that following inequality is satisfied
\begin{equation} \label{simplifiedep}
\epsilon  \leq 1 \wedge \left(4\tilde{c}(2) \sqrt{2\beta\pi/K_1}\int_0^{\tilde{b}}\exp\left(\left(s\sqrt{\beta K_1/8}+\sqrt{8/(\beta K_1)}\right)^2\right) \,\rmd s\right)^{-1},
\end{equation}
where $K_1 := L_1\mathbb{E}[\eta(X_0)]$, $\tilde{b}:=2\sqrt{2\tilde{c}(2)/\bar{c}(2)-1}$ and $\bar{b} := 2\sqrt{4\tilde{c}(2)(1+\bar{c}(2))/\bar{c}(2)-1}$. 

The constant $\hat{c}>0$ is given by $\hat{c} :=2(1+ \bar{b})\exp(\beta K_1 \bar{b}^2/8+2\bar{b})/\epsilon$.
\end{lemma}

Now, we are ready to prove our first main result, namely Theorem~\ref{main}.

\begin{proof}[\textbf{Proof of Theorem \ref{main}}] One notes that, by using $\lambda T >1/2$, Lemma \ref{convergencepart3} and Proposition \ref{contr}, for $t \in (nT, (n+1)T]$
\begin{align}
W_1(\mathcal{L}(\bar{\theta}^{\lambda}_t),\pi_{\beta}) 
&\leq W_1(\mathcal{L}(\bar{\theta}^{\lambda}_t),\mathcal{L}(Z^{\lambda}_t))+W_1(\mathcal{L}(Z^{\lambda}_t),\pi_{\beta}) \nonumber\\
& \leq(\bar{C}_{2,1}^{1/2} +\bar{C}_{2,2}^{1/2}+ \bar{C}_{2,3}+\bar{C}_{2,4})\sqrt{\lambda}(e^{-\dot{c}n/2}\mathbb{E}[V_4(\theta_0)] +1) \nonumber \\
&\quad +\hat{c} e^{-\dot{c}\lambda t} w_{1,2}(\theta_0, \pi_{\beta}) \nonumber\\
& \leq (\bar{C}_{2,1}^{1/2} +\bar{C}_{2,2}^{1/2}+ \bar{C}_{2,3}+\bar{C}_{2,4})\sqrt{\lambda}(e^{-\dot{c}n/2}\mathbb{E}[V_4(\theta_0)] +1) \nonumber\\
&\quad +\hat{c} e^{-\dot{c}\lambda t}\left[1+ \mathbb{E}[V_2(\theta_0)]+\int_{\mathbb{R}^d}V_2(\theta)\pi_{\beta}(d\theta)\right] \nonumber\\
\begin{split}\label{w1bd}
& \leq 2e^{-\dot{c}n/2} (\lambda_{\max}^{1/2}(\bar{C}_{2,1}^{1/2} +\bar{C}_{2,2}^{1/2}+ \bar{C}_{2,3}+\bar{C}_{2,4})+\hat{c})(1+\mathbb{E}[|\theta_0|^4])\\
&\quad +\hat{c}e^{-\dot{c}n/2}\left[1 +\int_{\mathbb{R}^d}V_2(\theta)\pi_{\beta}(d\theta)\right](1+\mathbb{E}[|\theta_0|^4])\\
&\quad +\sqrt{\lambda}(\bar{C}_{2,1}^{1/2} +\bar{C}_{2,2}^{1/2}+ \bar{C}_{2,3}+\bar{C}_{2,4}).
\end{split}
\end{align}
The above result implies, for any $n \in \mathbb{N}$,
\begin{equation}\label{w1estnT}
W_1(\mathcal{L}(\bar{\theta}^{\lambda}_{(n+1)T}),\pi_{\beta}) \leq C_1e^{-\dot{c}  (n+1)/2}(1+\mathbb{E}[|\theta_0|^4])+(C_2+C_3)\sqrt{\lambda},
\end{equation}
where
\begin{align}\label{mainthmconst}
\begin{split}
 C_1 &:= 2e^{\dot{c}/2} \left[ (\lambda_{\max}^{1/2}(\bar{C}_{2,1}^{1/2} +\bar{C}_{2,2}^{1/2}+ \bar{C}_{2,3}+\bar{C}_{2,4})+\hat{c})+\hat{c}\left(1+ \int_{\mathbb{R}^d}V_2(\theta)\pi_{\beta}(d\theta)\right)\right] \\
 &= O\left(e^{C_{\star}(1+d/\beta)(1+\beta)}\left(1+\frac{1}{1-e^{-\dot{c}}}\right)\right),\\
 C_2&: = \bar{C}_{2,1}^{1/2} +\bar{C}_{2,2}^{1/2} = O\left(1+\sqrt{\frac{d}{\beta}}\right), \\
  C_3 &:= \bar{C}_{2,3}+\bar{C}_{2,4} = O\left(e^{C_{\star}(1+d/\beta)(1+\beta)}\left(1+\frac{1}{1-e^{-\dot{c}}}\right)\right)
\end{split}
\end{align}
with $\dot{c}, \hat{c}$ given in Lemma \ref{contractionconst}, $\bar{C}_{2,1}$, $\bar{C}_{2,2}$ given in \eqref{barc2} (Lemma \ref{convergencepart1}), $\bar{C}_{2,3} $, $\bar{C}_{2,4}$ given in \eqref{mainthmc2} (Lemma \ref{convergencepart2}), $C_{\star}>0$ independent of $d, \beta, n$. One notes that the above estimate \eqref{w1estnT} is established for $(\bar{\theta}^{\lambda}_{(n+1)T})_{n \in \mathbb{N}}$. 
To obtain a non-asymptotic error bound for  $(\bar{\theta}^{\lambda}_{(n+1)})_{n \in \mathbb{N}}$, we set $(n+1)T$ to $n+1$ on the LHS of \eqref{w1estnT}, and set $n+1$ to $ (n+1)/T$ on the RHS of \eqref{w1estnT}. By using $\lambda (n+1) \leq (n+1)/T$, it follows that, for any $n \in \mathbb{N}$,
\[
W_1(\mathcal{L}(\theta^{\lambda}_{n+1}),\pi_{\beta}) \leq C_1e^{-\dot{c}\lambda (n+1)/2}(1+\mathbb{E}[|\theta_0|^4])+(C_2+C_3)\sqrt{\lambda}.
\]

Moreover, for $\varepsilon>0$, if we choose $\lambda$ and $n$ such that $\lambda\leq \lambda_{\max}$, $C_1e^{-\dot{c}\lambda n/2}(1+\mathbb{E}[|\theta_0|^4])\leq \varepsilon/2$, $(C_2+C_3)\sqrt{\lambda}\leq \varepsilon/2$, where $\lambda_{\max}$ is given in \eqref{eq:definition-lambda-max}, then $W_1(\mathcal{L}(\theta^{\lambda}_n),\pi_{\beta})\leq \varepsilon$. This implies $\lambda \leq \frac{\varepsilon^2}{4(C_2+C_3)^2} \wedge \lambda_{\max}$, $\lambda n \geq \frac{2}{\dot{c}}\ln \frac{2C_1(1+\mathbb{E}[|\theta_0|^4])}{\varepsilon}$. More precisely, by using \eqref{mainthmconst}, one obtains 
\[
n\geq \frac{C_{\star}e^{C_{\star}(1+d/\beta)(1+\beta)}}{\varepsilon^2\dot{c}}\left(1+\frac{1}{(1-e^{-\dot{c}})^2}\right)\ln\left( \frac{C_{\star}e^{C_{\star}(1+d/\beta)(1+\beta)}}{\varepsilon}\left(1+\frac{1}{1-e^{-\dot{c}}}\right)\right),
\]
where $\dot{c}$ is the contraction constant of the Langevin diffusion \eqref{sde} given explicitly in Lemma \ref{contractionconst}. 
\end{proof}

Next, we prove our second result Corollary~\ref{cw2}, i.e., a uniform bound in $W_2$ distance.
\begin{proof}[\textbf{Proof of Corollary \ref{cw2}}]  By using \eqref{vibd1} in Lemma \ref{convergencepart1}, Corollary \ref{convergencepart2w2} and Proposition \ref{contr}, one obtains
\begin{align*}
W_2(\mathcal{L}(\bar{\theta}^{\lambda}_t),\pi_{\beta})
&\leq W_2(\mathcal{L}(\bar{\theta}^{\lambda}_t),\mathcal{L}(Z^{\lambda}_t))+W_2(\mathcal{L}(Z^{\lambda}_t),\pi_{\beta})\\
& \leq W_2(\calL(\bar{\theta}^{\lambda}_t),\calL(\bar{\zeta}_t^{\lambda,n})) +W_2(\mathcal{L}(\bar{\zeta}_t^{\lambda,n}),\calL(Z_t^\lambda))+W_2(\mathcal{L}(Z^{\lambda}_t),\pi_{\beta})\\
& \leq \sqrt{\lambda} (e^{-an/4}\bar{C}_{2,1}\mathbb{E}[V_2(\theta_0)]  +\bar{C}_{2,2})^{1/2} \\
&\quad+\lambda^{1/4}(e^{-\dot{c}n/4}\bar{C}^*_{2,3}\mathbb{E}^{1/2}[V_4(\theta_0)] +\bar{C}^*_{2,4})\\
&\quad +\sqrt{2w_{1,2}(\mathcal{L}(Z^{\lambda}_t),\pi_{\beta})}\\
& \leq \lambda^{1/4}(\lambda_{\max}^{1/4}\bar{C}_{2,1}^{1/2} +\lambda_{\max}^{1/4}\bar{C}_{2,2}^{1/2} +\bar{C}^*_{2,3} + \bar{C}^*_{2,4})(e^{-\dot{c}n/4}\mathbb{E}[V_4(\theta_0)] +1)\\
&\quad +\hat{c}^{1/2} e^{-\dot{c}\lambda t/2} \sqrt{2w_{1,2}(\theta_0, \pi_{\beta})},
\end{align*}
Further calculations yield,
\begin{align*}
&W_2(\mathcal{L}(\bar{\theta}^{\lambda}_t),\pi_{\beta})\\
& \leq \lambda^{1/4}(\lambda_{\max}^{1/4}\bar{C}_{2,1}^{1/2} +\lambda_{\max}^{1/4}\bar{C}_{2,2}^{1/2} +\bar{C}^*_{2,3} + \bar{C}^*_{2,4}) (e^{-\dot{c}n/4}\mathbb{E}[V_4(\theta_0)] +1) \\
&\quad +\sqrt{2}\hat{c}^{1/2} e^{-\dot{c}\lambda t/2}\left(1+ \mathbb{E}[V_2(\theta_0)]+\int_{\mathbb{R}^d}V_2(\theta)\pi_{\beta}(d\theta)\right)^{1/2}\\
& \leq 2e^{-\dot{c}n/4} (\lambda_{\max}^{1/2}(\bar{C}_{2,1}^{1/2} + \bar{C}_{2,2}^{1/2} )+\lambda_{\max}^{1/4}(\bar{C}^*_{2,3} + \bar{C}^*_{2,4})+\sqrt{2}\hat{c}^{1/2})(1+\mathbb{E}[|\theta_0|^4])\\
&\quad +\sqrt{2}\hat{c}^{1/2}e^{-\dot{c}n/4}\left[1 +\int_{\mathbb{R}^d}V_2(\theta)\pi_{\beta}(d\theta)\right]\\
&\quad + \lambda^{1/4}(\lambda_{\max}^{1/4}\bar{C}_{2,1}^{1/2} +\lambda_{\max}^{1/4}\bar{C}_{2,2}^{1/2} +\bar{C}^*_{2,3} + \bar{C}^*_{2,4}),
\end{align*}
where the last inequality holds due to $\lambda T >1/2$. Thus, for any $n \in \mathbb{N}$, it follows that
\[
W_2(\mathcal{L}(\theta^{\lambda}_n),\pi_{\beta}) \leq C_4 \rme^{-\dot{c}\lambda n/4}\E[|\theta_{0}|^{4}+1] +(C_5+C_6)\lambda^{1/4}
\]
where
\begin{align}\label{cw2const}
\begin{split}
C_4 &:= 2 \left(\lambda_{\max}^{1/2}(\bar{C}_{2,1}^{1/2} + \bar{C}_{2,2}^{1/2} )+\lambda_{\max}^{1/4}(\bar{C}^*_{2,3} + \bar{C}^*_{2,4})+\sqrt{2}\hat{c}^{1/2}\right)\\
&\quad +\sqrt{2}\hat{c}^{1/2}\left(1+ \int_{\mathbb{R}^d}V_2(\theta)\pi_{\beta}(d\theta)\right) \\
&= O\left(e^{C_{\star}(1+d/\beta)(1+\beta)}\left(1+\frac{1}{1-e^{-\dot{c}/2}}\right)\right)\\
 C_5 &:= \lambda_{\max}^{1/4}\bar{C}_{2,1}^{1/2} +\lambda_{\max}^{1/4}\bar{C}_{2,2}^{1/2} =O\left(1+\sqrt{\frac{d}{\beta}}\right), \\
  C_6 &:= \bar{C}^*_{2,3} + \bar{C}^*_{2,4} = O\left(e^{C_{\star}(1+d/\beta)(1+\beta)}\left(1+\frac{1}{1-e^{-\dot{c}/2}}\right)\right),
\end{split}
\end{align}
with $\dot{c}, \hat{c}$ given in Lemma \ref{contractionconst}, $\bar{C}_{2,1}$, $\bar{C}_{2,2}$ given in \eqref{barc2} (Lemma \ref{convergencepart1}),  $\bar{C}^*_{2,3} $, $\bar{C}^*_{2,4}$ given in \eqref{mainthmcstar2} (Lemma \ref{convergencepart2w2}), $C_{\star}>0$ independent of $d, \beta, n$. 

Moreover, for $\varepsilon>0$, if we choose $\lambda$ and $n$ such that, $\lambda\leq \lambda_{\max}$, $C_4 \rme^{-\dot{c}\lambda n/4}\E[|\theta_{0}|^{4}+1] \leq \varepsilon/2$,  $(C_5+C_6)\lambda^{1/4}\leq \varepsilon/2$, where $\lambda_{\max}$ is given in \eqref{eq:definition-lambda-max}, then $W_2(\mathcal{L}(\theta^{\lambda}_n),\pi_{\beta})\leq \varepsilon$. This implies
$\lambda \leq \frac{\varepsilon^4}{16(C_5+C_6)^4} \wedge \lambda_{\max}$, $ \lambda n \geq \frac{4}{\dot{c}}\ln \frac{2C_4(1+\mathbb{E}[|\theta_0|^4])}{\varepsilon}$. More precisely, by using \eqref{cw2const}, one obtains 
\[
n\geq \frac{C_{\star}e^{C_{\star}(1+d/\beta)(1+\beta)}}{\varepsilon^4\dot{c}}\left(1+\frac{1}{(1-e^{-\dot{c}/2})^4}\right)\ln\left( \frac{C_{\star}e^{C_{\star}(1+d/\beta)(1+\beta)}}{\varepsilon}\left(1+\frac{1}{1-e^{-\dot{c}/2}}\right)\right),
\] 
where $\dot{c}$ is the contraction constant of the Langevin diffusion \eqref{sde} given explicitly in Lemma \ref{contractionconst}. 
\end{proof}

Finally, we move on to prove our result on nonconvex optimization, namely, Corollary~\ref{eer}.

\begin{proof}[\textbf{Proof of Corollary \ref{eer}}] To obtain an upper bound for the expected excess risk $\mathbb{E}[U( \theta^{\lambda}_n)] - \inf_{\theta \in \mathbb{R}^d} U(\theta) $, one considers the following splitting
\begin{equation}\label{eersplitting}
\mathbb{E}[U( \theta^{\lambda}_n)] - \inf_{\theta \in \mathbb{R}^d} U(\theta) = \left( \mathbb{E}[U( \theta^{\lambda}_n)]  -  \mathbb{E}[U(Z_{\infty})]\right) + \left( \mathbb{E}[U(Z_{\infty})]- \inf_{\theta \in \mathbb{R}^d} U(\theta) \right),
\end{equation}
where $Z_{\infty}\sim \pi_{\beta}$ with $\pi_{\beta}$ defined in \eqref{pibetaexp}. By using \cite[Lemma~3.5]{raginsky}, Remark \ref{rem:BoundsOnH}, Lemma \ref{lem:moment_SGLD_2p}, and Corollary \ref{cw2}, the first term on the RHS of \eqref{eersplitting} can be bounded by
\begin{align*}
&\mathbb{E}[U( \theta^{\lambda}_n)]  -  \mathbb{E}[U(Z_{\infty})] \\
& \leq \left(L_1 \mathbb{E}[\eta(X_0)](\mathbb{E}[|\theta_0|^2]+c_1(\lambda_{\max}+a^{-1})) +L_2\mathbb{E}[\bar{\eta}(X_0)]+H_{\star}\right)W_2(\mathcal{L}(\theta^{\lambda}_n),\pi_{\beta})\\
&\leq \left(L_1 \mathbb{E}[\eta(X_0)](\mathbb{E}[|\theta_0|^2]+c_1(\lambda_{\max}+a^{-1})) +L_2\mathbb{E}[\bar{\eta}(X_0)]+H_{\star}\right)\\
&\quad \times\left( C_4 \rme^{-\dot{c}\lambda n/4}\E[|\theta_{0}|^{4}+1] +(C_5+C_6)\lambda^{1/4}\right)\\
&\leq C^{\sharp}_1e^{-\dot{c}\lambda n/4}+C^{\sharp}_2\lambda^{1/4},
\end{align*}
where
\begin{align}\label{eerconst}
\begin{split}
C^{\sharp}_1 &: = C_4\left(L_1 \mathbb{E}[\eta(X_0)](\mathbb{E}[|\theta_0|^2]+c_1(\lambda_{\max}+a^{-1})) +L_2\mathbb{E}[\bar{\eta}(X_0)]+H_{\star}\right)\E[|\theta_{0}|^{4}+1], \\
C^{\sharp}_2 & :=(C_5+C_6)\left(L_1 \mathbb{E}[\eta(X_0)](\mathbb{E}[|\theta_0|^2]+c_1(\lambda_{\max}+a^{-1})) +L_2\mathbb{E}[\bar{\eta}(X_0)]+H_{\star}\right),
\end{split}
\end{align}
with $\dot{c}$ given in \eqref{dotc}, $C_4, C_5, C_6$ given in \eqref{cw2const} and $c_1$ given in \eqref{2ndmomentconst}. Moreover, the second term on the RHS of \eqref{eersplitting} can be estimated by using \cite[Proposition~3.4]{raginsky}, which gives, $\mathbb{E}[U(Z_{\infty})]- \inf_{\theta \in \mathbb{R}^d} U(\theta) \leq  C^{\sharp}_3$,
where
\begin{equation}\label{eerconst2}
C^{\sharp}_3 := \frac{d}{2\beta}\log\left(\frac{e L_1\mathbb{E}[\eta(X_0)]}{a}\left(\frac{b\beta}{d}+1\right)\right).
\end{equation}
Finally, one obtains 
\[
\mathbb{E}[U( \theta^{\lambda}_n)] - \inf_{\theta \in \mathbb{R}^d} U(\theta)  \leq  C^{\sharp}_1e^{- \dot{c}\lambda n/4}+C^{\sharp}_2\lambda^{1/4}+C^{\sharp}_3,
\] 
where  $\dot{c}$ is given in \eqref{dotc}, and 
\begin{align*}
C^{\sharp}_1 &= O\left(e^{C_{\star}(1+d/\beta)(1+\beta)}\left(1+\frac{1}{1-e^{-\dot{c}/2}}\right)\right),\\
C^{\sharp}_2 &= O\left(e^{C_{\star}(1+d/\beta)(1+\beta)}\left(1+\frac{1}{1-e^{-\dot{c}/2}}\right)\right),\\
C^{\sharp}_3 &= O((d/\beta)\log(C_{\star}(\beta/d+1)))
\end{align*}
with 
$C_{\star}>0$ a constant independent of $n, d, \beta$.

Moreover, for $\varepsilon>0$, if we choose $\beta$ such that $C^{\sharp}_3 \leq \varepsilon/3$, then choose $\lambda$ such that $\lambda\leq \lambda_{\max}$ with $\lambda_{\max}$ given in \eqref{eq:definition-lambda-max} and $C^{\sharp}_2\lambda^{1/4}\leq \varepsilon/3$, and finally choose $n$ such that $C^{\sharp}_1e^{- \dot{c}\lambda n/4} \leq \varepsilon/3$, consequently, we obtain $\mathbb{E}[U( \theta^{\lambda}_n)] - \inf_{\theta \in \mathbb{R}^d} U(\theta) \leq \varepsilon$. This implies $\beta \geq \beta_{\varepsilon} \vee \frac{3d}{\varepsilon}\log\left(\frac{e L_1\mathbb{E}[\eta(X_0)]}{ad}\left(b+1\right)\left(d+1\right)\right) $, where $\beta_{\varepsilon}$ is the root of the function $f^\sharp(\beta) = \frac{\log\left( \beta+1\right)}{ \beta}-\frac{\varepsilon}{3d}$, $\beta>0$, i.e. $f^\sharp(\beta_{\varepsilon})=0 $. Indeed, since 
\[
C^{\sharp}_3 \leq \frac{d}{2\beta}\log\left(\frac{e L_1\mathbb{E}[\eta(X_0)]}{ad}\left(b+1\right)\left(d+1\right)\left(\beta+1\right)\right),
\]
by setting $\frac{d}{2\beta}\log\left(\frac{e L_1\mathbb{E}[\eta(X_0)]}{ad}\left(b+1\right)\left(d+1\right)\right) \leq \varepsilon/6$ and $\frac{d}{2\beta}\log\left( \beta+1\right)  \leq\varepsilon/6$, one obtains $C^{\sharp}_3 \leq \varepsilon/3$. Noticing that $\frac{\log(\beta+1)}{\beta}$ is decreasing in $\beta$ yields the desired result. Furthermore, calculations yield $\lambda \leq \frac{\varepsilon^4}{81(C^{\sharp}_2)^4} \wedge \lambda_{\max}$, and $ \lambda n \geq \frac{4}{\dot{c}}\ln \frac{3C^{\sharp}_1}{\varepsilon}$. More precisely, $n\geq \frac{C_{\star}e^{C_{\star}(1+d/\beta)(1+\beta)}}{\varepsilon^4\dot{c}}\left(1+\frac{1}{(1-e^{-\dot{c}/2})^4}\right)\ln\left( \frac{C_{\star}e^{C_{\star}(1+d/\beta)(1+\beta)}}{\varepsilon}\left(1+\frac{1}{1-e^{-\dot{c}/2}}\right)\right)$, 
where $\dot{c}$ is the contraction constant of the Langevin diffusion \eqref{sde} given explicitly in Lemma \ref{contractionconst}. 
\end{proof}

\section{Conclusions}
We have provided non-asymptotic estimates for the SGLD which explicitly bounds the error between the target measure and the law of the SGLD in Wasserstein-$1$ and $2$ distances. These results further allow us to establish a non-asymptotic error bound for the expected excess risk. Moreover, the theoretical findings enable us to obtain theoretical guarantees for fundamental problems in machine learning and in financial mathematics: Nonasymptotic error bounds for \textit{nonconvex optimization} problems. We have shown that our assumptions are verifiable for a large class of practical problems. In particular, we demonstrate this by providing 
 two important applications: (i) \textit{variational inference for Bayesian logistic regression} (VI), 
 (ii) \textit{index tracking optimization}. We believe that our results provide a detailed understanding of the sampling behaviour of SGLD even when it is examined within the context of nonconvex optimization.




\appendix 

\section*{Appendix}
\section{Additional Lemmata}\label{app}
\begin{lemma}\label{lem:boundedVariance} Let Assumptions \ref{iid}, \ref{loclip} and \ref{assum:dissipativity} hold. For any $t  \in (nT, (n+1)T]$, $n \in \mathbb{N}$ and $k = 1, \dots, K+1$, $K+1 \leq T$, one obtains
\[
\E\left[\left|h(\bar{\zeta}_t^{\lambda,n}) - H(\bar{\zeta}_t^{\lambda,n},X_{nT+k})\right|^2\right] \leq e^{-a\lambda t/2}\bar{\sigma}_Z\mathbb{E}[V_2(\theta_0)]+\tilde{\sigma}_Z,
\]
where
\begin{align}\label{sigmaZ}
\begin{split}
\bar{\sigma}_Z 		&:=8L_2^2\hat{\sigma}_Z, \quad \tilde{\sigma}_Z 	:=8L_2^2\hat{\sigma}_Z(3\mathrm{v}_2(\overline{M}_2)+c_1(\lambda_{\max}+a^{-1})+1),\\
\hat{\sigma}_Z		& := \E[(\eta(X_0)+\eta(\E[X_0]))^2|X_0 - \E[X_0]|^2].
\end{split}
\end{align} 
\end{lemma}
\begin{proof} Recall  $\mathcal{H}_t = \mathcal{F}^{\lambda}_{\infty} \vee \mathcal{G}_{\lfloor t \rfloor}$. One notices that
\begin{align*}
&\E\left[\left|h(\bar{\zeta}_t^{\lambda,n}) - H(\bar{\zeta}_t^{\lambda,n},X_{nT+k})\right|^2\right] \\
& =\E\left[\E\left[\left.\left|h(\bar{\zeta}_t^{\lambda,n}) - H(\bar{\zeta}_t^{\lambda,n},X_{nT+k})\right|^2\right|\mathcal{H}_{nT}\right]\right] \\
&=\E\left[\E\left[\left.\left|\E\left[\left.H(\bar{\zeta}_t^{\lambda,n}, X_{nT+k})\right|\mathcal{H}_{nT}\right] - H(\bar{\zeta}_t^{\lambda,n},X_{nT+k})\right|^2\right|\mathcal{H}_{nT}\right]\right] \\
&\leq 4\E\left[\E\left[\left.\left| H(\bar{\zeta}_t^{\lambda,n},X_{nT+k})- H(\bar{\zeta}_t^{\lambda,n}, \E\left[\left. X_{nT+k}\right|\mathcal{H}_{nT}\right])\right|^2\right|\mathcal{H}_{nT} \right]\right] \\
&\leq 4L_2^2\hat{\sigma}_Z\E\left[\left(1+\left|\bar{\zeta}_t^{\lambda,n} \right|\right)^2\right],
\end{align*}
where the first inequality holds due to Lemma \ref{lem:useful-conditional-expectation} and $\hat{\sigma}_Z := \E[(\eta(X_0)+\eta(\E[X_0]))^2|X_0 - \E[X_0]|^2]$. Then, by using Lemma \ref{zetaprocessmoment}, one obtains
\[
\E\left[\left|h(\bar{\zeta}_t^{\lambda,n}) - H(\bar{\zeta}_t^{\lambda,n},X_{nT+k})\right|^2\right] \leq 8L_2^2\hat{\sigma}_Z\E\left[V_2(\bar{\zeta}_t^{\lambda,n})\right] \leq e^{-a\lambda t/2}\bar{\sigma}_Z \E[V_2(\theta_0)]+\tilde{\sigma}_Z ,
\]
where $\bar{\sigma}_Z := 8L_2^2\hat{\sigma}_Z$ and $\tilde{\sigma}_Z :=8L_2^2\hat{\sigma}_Z(3\mathrm{v}_2(\overline{M}_2)+c_1(\lambda_{\max}+a^{-1})+1)$.
\end{proof}
\begin{lemma} \label{onestepest} Let Assumptions \ref{iid}, \ref{loclip} and \ref{assum:dissipativity} hold. For any $t >0$, one obtains
\begin{align*}
\E\left[ \left| \bar{\theta}^{\lambda}_t - \bar{\theta}^{\lambda}_{\floor{t}} \right|^2 \right] \leq \lambda(e^{-a\lambda \floor{t}}\bar{\sigma}_Y \mathbb{E}[V_2(\theta_0)] + \tilde{\sigma}_Y),
\end{align*}
where
\begin{align}\label{sigmaY}
\begin{split}
\bar{\sigma}_Y 		&: = 2\lambda_{\max}L_1^2  \mathbb{E}[\eta^2(X_0)]\\
\tilde{\sigma}_Y 	&:= 2\lambda_{\max}L_1^2   \mathbb{E}[\eta^2(X_0)]c_1(\lambda_{\max}+a^{-1})+4\lambda_{\max}L_2^2\mathbb{E}[\bar{\eta}^2(X_0)]\\
						&\quad  +4\lambda_{\max} H_\star^2+2d\beta^{-1}.
\end{split}
\end{align}
\end{lemma}
\begin{proof} For any $t>0$, we write the difference $\left| \bar{\theta}^{\lambda}_t - \bar{\theta}^{\lambda}_{\floor{t}}\right|$ and use $(a+b)^2 \leq 2a^2 + 2b^2$ which yields
\begin{align*}
\E\left[\left| \bar{\theta}^{\lambda}_t - \bar{\theta}^{\lambda}_{\floor{t}} \right|^2\right] &= \E\left[\left| -\lambda \int_{\floor{t}}^t H(\bar{\theta}^{\lambda}_{\floor{t}},X_{\ceil{t}}) \rmd s + \sqrt{\frac{2\lambda}{\beta}} ( \tilde{B}_t^{\lambda}-\tilde{B}_{\floor{t}}^{\lambda}) \right|^2\right]\\
& \leq \lambda^2\mathbb{E}\left[\left(L_1\eta(X_{\ceil{t}})|\bar{\theta}^{\lambda}_{\floor{t}}|+L_2\bar{\eta}(X_{\ceil{t}})+H_\star\right)^2\right]+ 2d\lambda\beta^{-1},
\end{align*}
where the inequality holds due to Remark \ref{rem:BoundsOnH} and by applying Lemma \ref{lem:moment_SGLD_2p}, one obtains
\begin{align*}
\E\left[\left| \bar{\theta}^{\lambda}_t - \bar{\theta}^{\lambda}_{\floor{t}} \right|^2\right]
& \leq 2\lambda^2 L_1^2 \mathbb{E}[\eta^2(X_0)] \mathbb{E}[|\bar{\theta}^{\lambda}_{\floor{t}}|^2]+4\lambda^2 L_2^2\mathbb{E}[\bar{\eta}^2(X_0)]  +4\lambda^2H_\star^2 +2d\lambda\beta^{-1}\\
&\leq \lambda((1-a\lambda)^{\floor{t}}\bar{\sigma}_Y \mathbb{E}[V_2(\theta_0)] + \tilde{\sigma}_Y),
\end{align*}
where $\bar{\sigma}_Y: = 2\lambda_{\max}L_1^2  \mathbb{E}[\eta^2(X_0)] $ and $\tilde{\sigma}_Y := 2\lambda_{\max}L_1^2   \mathbb{E}[\eta^2(X_0)]c_1(\lambda_{\max}+a^{-1})+4\lambda_{\max}L_2^2\mathbb{E}[\bar{\eta}^2(X_0)] +4\lambda_{\max} H_\star^2+2d\beta^{-1}$.
\end{proof}

\begin{lemma}\label{lem:useful-conditional-expectation}
Let $\mathcal{G},\mathcal{H}\subset\mathcal{F}$ be sigma-algebras. Let $p\geq 1$. Let $X,Y$ be $\mathbb{R}^d$-valued random vectors in $L^p$ such that $Y$ is measurable with respect to $\mathcal{H}\vee\mathcal{G}$. Then, $\E^{1/p}\left[\left.\left|X-\E[X|\mathcal{H \vee G}]\right|^p\right|\mathcal{G}\right]\leq 2\E^{1/p}\left[\left.\left|X-Y\right|^p\right|\mathcal{G}\right]$.
\end{lemma}
\begin{proof} See \cite[Lemma~{6.1}]{4}.
\end{proof}

\section{Proofs of the results in \ref{examplemain}}\label{proof:examplemain}
\begin{proof}[\textbf{Proof of Proposition~\ref{prop:VI}}]\label{proof:prop:VI}
By using \eqref{expressionUVI}, it can be shown by direct calculations that $H$ defined in \eqref{HexpressionVI} satisfies Assumption \ref{iid}. 

One notes that \eqref{HexpressionVI} can be rewritten as
\begin{align*}
H(\theta, u)
& =  \sum_{i = 1}^nH_i(\theta, u) \\
&= \sum_{i = 1}^n\left(\frac{1}{n}\left(\frac{\theta}{2} + \frac{u}{4}-\frac{3}{4}\hat{a}+ \frac{\hat{a}}{4}\left(\frac{7}{1+e^{2\hat{a}^{\mathsf{T}}(u/4+\theta)}}-\frac{1}{1+e^{2\hat{a}^{\mathsf{T}}(u/4-\theta)}}\right)\right.\right.\\
&\hspace{5em}\left. \left.  -\frac{7(u+8\theta)}{16(1+e^{2(\theta^{\mathsf{T}}u/4+|\theta|^2)})}-\frac{u-8\theta}{16(1+e^{2(\theta^{\mathsf{T}}u/4-|\theta|^2)})}\right)\right.\\
&\hspace{8em} +\left. \frac{1}{8}\left( -6z_iy_i+\frac{7z_i}{1+e^{-z_i^{\mathsf{T}}(u/4+\theta)}}-\frac{z_i}{1+e^{-z_i^{\mathsf{T}}(u/4-\theta)}}\right)\right),
\end{align*}
where $H_i: \mathbb{R}^d \times \mathbb{R}^d \rightarrow \mathbb{R}^d$ for each $i = 1, \dots, n$.
To verify Assumption \ref{assum:dissipativity}, which is the (local) dissipativity condition, one calculates, for any $\theta \in \mathbb{R}^d$, $u \in \mathbb{R}^d$
\begin{align*}
\theta^\mathsf{T} H_i(\theta, u) &= \frac{1}{n}\left(\frac{|\theta|^2}{2} + \frac{u^\mathsf{T} \theta}{4}-\frac{3}{4}\hat{a}^\mathsf{T} \theta+ \frac{\hat{a}^\mathsf{T} \theta}{4}\left(\frac{7}{1+e^{2\hat{a}^{\mathsf{T}}(u/4+\theta)}}-\frac{1}{1+e^{2\hat{a}^{\mathsf{T}}(u/4-\theta)}}\right) \right.\\
&\hspace{4em} \left. -\frac{7(u^\mathsf{T} \theta+8|\theta|^2)}{16(1+e^{2(\theta^{\mathsf{T}}u/4+|\theta|^2)})}-\frac{u^\mathsf{T} \theta-8|\theta|^2}{16(1+e^{2(\theta^{\mathsf{T}}u/4-|\theta|^2)})}\right)\\
&\quad +\frac{1}{8}\left( -6y_iz_i^{\mathsf{T}}\theta +\frac{7z_i^\mathsf{T}\theta}{1+e^{-z_i^{\mathsf{T}}(u/4+\theta)}}-\frac{z_i^\mathsf{T}\theta}{1+e^{-z_i^{\mathsf{T}}(u/4-\theta)}}\right)\\
&\geq \frac{1}{n}\left(\frac{|\theta|^2}{2} -  \frac{3|u^\mathsf{T} \theta|}{4}- \frac{11}{4}|\hat{a}^\mathsf{T} \theta| \right)-\frac{7}{4}|z_i^\mathsf{T}\theta |-\frac{7}{4}\\
&\geq \frac{1}{4n}|\theta|^2 - \frac{1}{n}\left(\frac{9| u|^2}{4}+\frac{121}{4}|\hat{a}|^2\right)-\frac{49n}{8}|z_i|^2 - \frac{7}{4},
\end{align*}
which implies
\[
  \theta^\mathsf{T} H(\theta, u)  \geq \frac{1}{4}|\theta|^2 -  \left(\frac{9| u|^2}{4}+\frac{121}{4}|\hat{a}|^2\right)-\frac{49n}{8}\sum_{i = 1}^n|z_i|^2- \frac{7n}{4}.
\]
Thus the (local) dissipativity condition holds with $A(u) = \mathbf{I}_d /4$ and $\hat{b}(u) =   (9| u|^2/4+121|\hat{a}|^2/4)+49n\sum_{i = 1}^n|z_i|^2/8+7n/4$. 

As for the Lipschitz conditions in Assumption \ref{loclip}, one notices that $1+e^{2(\theta^{\mathsf{T}}u/4+|\theta|^2)}  = e^{-|u|^2/32}( e^{|u|^2/32}+e^{2|\theta +u/8|^2})$, then
\begin{align*}
\nabla_{\theta} H_i(\theta, u) &= \frac{1}{n}\left(\frac{\mathbf{I}_d}{2}-\frac{\hat{a}\hat{a}^{\mathsf{T}}}{2}\left(\frac{7e^{2\hat{a}^{\mathsf{T}}(u/4+\theta)}}{(1+e^{2\hat{a}^{\mathsf{T}}(u/4+\theta)})^2}+\frac{e^{2\hat{a}^{\mathsf{T}}(u/4-\theta)}}{(1+e^{2\hat{a}^{\mathsf{T}}(u/4-\theta)})^2}\right)\right.\\ 
&\hspace{4em} -\left(\frac{7\mathbf{I}_d}{2(1+e^{2(\theta^{\mathsf{T}}u/4+|\theta|^2)})}-\frac{\mathbf{I}_d}{2(1+e^{2(\theta^{\mathsf{T}}u/4-|\theta|^2)})}\right)\\
&\hspace{4em} \left. +\left(\frac{7e^{|u|^2/32}e^{2|\theta +u/8|^2}(u+8\theta)(u^{\mathsf{T}}+8\theta^{\mathsf{T}})}{32(e^{|u|^2/32}+e^{2|\theta +u/8|^2})^2} \right.\right.\\
&\hspace{7em} \left. \left. +\frac{e^{|u|^2/32}e^{2|\theta -u/8|^2}(u-8\theta)(u^{\mathsf{T}}-8\theta^{\mathsf{T}})}{32(e^{|u|^2/32}+e^{2|\theta -u/8|^2})^2}\right)\right)\\
&\quad +\frac{z_iz_i^{\mathsf{T}}}{8}\left( \frac{7e^{-z_i^{\mathsf{T}}(u/4+\theta)}}{(1+e^{-z_i^{\mathsf{T}}(u/4+\theta)})^2}+\frac{e^{-z_i^{\mathsf{T}}(u/4-\theta)}}{(1+e^{-z_i^{\mathsf{T}}(u/4-\theta)})^2}\right),
\end{align*}
which implies Assumption \ref{loclip} holds with $L_1=1$ and $\eta(u) = 9/2+8e^{|u|^2/32}+\sum_{i = 1}^n|z_i|^2+4|\hat{a}|^2+3|u|^2/8$. On the other hand,
\begin{align*}
\nabla_{u} H_i(\theta, u) &= \frac{1}{n}\left( \frac{1}{4}\mathbf{I}_d-\frac{\hat{a}\hat{a}^{\mathsf{T}}}{8}\left(\frac{7e^{2\hat{a}^{\mathsf{T}}(u/4+\theta)}}{(1+e^{2\hat{a}^{\mathsf{T}}(u/4+\theta)})^2}-\frac{e^{2\hat{a}^{\mathsf{T}}(u/4-\theta)}}{(1+e^{2\hat{a}^{\mathsf{T}}(u/4-\theta)})^2}\right)\right.\\
&\hspace{3em} -\left(\frac{7\mathbf{I}_d }{16(1+e^{2(\theta^{\mathsf{T}}u/4+|\theta|^2)})}+\frac{\mathbf{I}_d }{16(1+e^{2(\theta^{\mathsf{T}}u/4-|\theta|^2)})}\right)\\
&\hspace{3em} \left. +\left(\frac{7e^{|u|^2/32}e^{2|\theta +u/8|^2}(u+8\theta)(-u^{\mathsf{T}}/16+(u^{\mathsf{T}}/2+4\theta^{\mathsf{T}})/8)}{16(e^{|u|^2/32}+e^{2|\theta +u/8|^2})^2}\right.\right.\\
&\hspace{5em}\left.\left. +\frac{e^{|u|^2/32}e^{2|\theta -u/8|^2}(u-8\theta)( u^{\mathsf{T}}/16+(4\theta^{\mathsf{T}}-u^{\mathsf{T}}/2)/8)}{16(e^{|u|^2/32}+e^{2|\theta -u/8|^2})^2}\right)\right)\\
&\quad +\frac{z_iz_i^{\mathsf{T}}}{32}\left( \frac{7e^{-z_i^{\mathsf{T}}(u/4+\theta)}}{(1+e^{-z_i^{\mathsf{T}}(u/4+\theta)})^2}-\frac{e^{-z_i^{\mathsf{T}}(u/4-\theta)}}{(1+e^{-z_i^{\mathsf{T}}(u/4-\theta)})^2}\right),
\end{align*}
which implies Assumption \ref{loclip} holds with $L_2 = 1/4$ and $\eta(u) = 9/2+8e^{|u|^2/32}+\sum_{i = 1}^n|z_i|^2+4|\hat{a}|^2+3|u|^2/8$.
\end{proof}

\begin{proof}[\textbf{Proof of Proposition~\ref{prop:IndTrasg}}]\label{proof:prop:IndTrasg}
First, we show that the objective function $U$ defined in \eqref{eq:itob} is not necessarily convex. We consider the case where $N = 2$, and similar arguments can be applied for $N \geq 2$. By using \eqref{eq:itobgrad}, the Hessian matrix of $U$, denoted by $\nabla^2 U$, is given by:
\[\nabla^2 U(\theta) = 
\begin{pmatrix}
2\hat{\eta} +M(\theta)	& - M(\theta) \\
 - M(\theta)			& 2\hat{\eta} +M(\theta)
\end{pmatrix}, \quad \theta \in \R^2
\] 
where $M(\theta):=2g_1^2(\theta)g_2^2(\theta)\E[(X_2-X_1)^2]+2g_1(\theta)g_2(\theta)(1-2g_1(\theta))\E[(Y - g_1(\theta)X_1-g_2(\theta)X_2)(X_2-X_1)]$.
Then, for any $v=(v_1, v_2) \in \R^2 \setminus \{(0,0)\}$, it follows that
\[
\langle v, \nabla^2 U(\theta) v\rangle = 2\hat{\eta} |v|^2+M(\theta)(v_1-v_2)^2,
\]
which is not necessarily nonnegative for all $\theta \in \R^2$. To see this, we consider the following example. Let 
\begin{align*}
&\E[X_1] = 0.03, \quad \E[X_2] = 0.04, \quad \E[Y] = 0.033, \\
&\Var(X_1) = 5 \times 10^{-5}, \quad \Var(X_2) = 2 \times 10^{-4}, \quad  \Var(Y) = 5.5 \times 10^{-5},\\
&\Cov(X_1, X_2) = 10^{-5}, \quad \Cov(X_1, Y) = 5 \times 10^{-6}, \quad \Cov(X_2, Y) = -9\times 10^{-5},
\end{align*}
which implies $\E[YX_2] = 1.23\times 10^{-3}$, $ \E[X_1X_2] =1.21\times 10^{-3}$, $\E[X_1Y] = 9.95\times 10^{-4}$, $\E[X_1^2] = 9.5\times 10^{-4}$.
Then, set $\hat{\eta} = 10^{-6}, v= (1,0)$. For $\theta = (1, \ln 2)$, i.e., $g_1(\theta) = 1/3$, one obtains,
\begin{align*}
M(\theta)
& = 2g_1 (\theta)g_2^2(\theta)(3g_1(\theta)-1)\E[(X_2-X_1)^2]\\
&\quad+2g_1(\theta)g_2(\theta)(1-2g_1(\theta))\E[(Y- X_1)(X_2-X_1)]\\
& = 2g_1(\theta)g_2(\theta)(1-2g_1(\theta))\E[(YX_2- X_1X_2-YX_1+X_1^2)]  = -\frac{1}{270000},
\end{align*}
which indicates $\langle v, \nabla^2 U(\theta) v\rangle = -\frac{23}{13500000}<0$. In addition, for $\theta = (1,1)$, i.e. $g_1(\theta ) = 1/2$, one obtains
\begin{align*}
M(\theta)
& = 2g_1 (\theta)g_2^2(\theta)(3g_1(\theta)-1)\E[(X_2-X_1)^2]\\
&\quad+2g_1(\theta)g_2(\theta)(1-2g_1(\theta))\E[(Y- X_1)(X_2-X_1)]\\
& = 2g_1 (\theta)g_2^2(\theta)(3g_1(\theta)-1)\E[(X_2-X_1)^2]\geq 0.
\end{align*}
which implies $\langle v, \nabla^2 U(\theta) v\rangle \geq 0$. Thus, one concludes that $U$ is in general nonconvex. 

Next, we prove that the stochastic gradient $H$ given in \eqref{eqn:itsg} satisfies Assumptions~\ref{iid}, \ref{loclip}, and \ref{assum:dissipativity}. Recall the explicit expressions for $H_m$, $m = 1, \dots, N$ are given as follows:
\begin{align*}
H_m(\theta,z) 
&= 2\hat{\eta} \theta_m +2\left(y -  \sum_{i = 1}^N g_i(\theta) x_{i}\right) g_m(\theta)\sum_{i \neq m}^N g_i(\theta)(x_i- x_m).
\end{align*}
It is easily checkable that Assumption \ref{iid} holds. To see Assumption \ref{assum:dissipativity} is satisfied, one calculates the following: for any $\theta \in \R^N, z \in \R^{N+1}$, 
\begin{align*}
\langle \theta, H(\theta, z) \rangle 
&=\sum_{m = 1}^N \theta_m H_m(\theta, z) \geq \hat{\eta} |\theta|^2 - \hat{\eta}^{-1}N\left( \left(|y|+\sum_{i = 1}^N|x_i|\right)\sum_{i \neq m}^N(|x_i|+|x_m|)\right)^2.
\end{align*}
Assumption \ref{loclip} is also satisfied. Indeed, for any $m = 1, \dots, N$, $\theta, \theta' \in \R^N, z \in \R^{N+1}$,  it follows that
\begin{align*}
|H_m(\theta, z) - H_m(\theta', z)|
&\leq 2\hat{\eta}|\theta_m-\theta'_m|\\
&\quad +2\left| \left(y -  \sum_{i = 1}^N g_i(\theta) x_{i}\right) g_m(\theta)\sum_{i \neq m}^N g_i(\theta)(x_i- x_m)\right.\\
&\qquad - \left.  \left(y -  \sum_{i = 1}^N g_i(\theta') x_{i}\right) g_m(\theta')\sum_{i \neq m}^N g_i(\theta')(x_i- x_m) \right|\\
&\leq 6\sqrt{N}\left(\hat{\eta}+ \left( |y|+\sum_{i = 1}^N|x_i|\right)\sum_{i \neq m}^N(|x_i|+|x_m|) \right)|\theta - \theta'|,
\end{align*}
where the last inequality holds due to the following: for any $m = 1, \dots, N$, $\theta, \theta' \in \R^N$.
\[
|g_m(\theta)-g_m(\theta')| \leq \sqrt{N}|\theta-\theta'|.
\]
Then, one obtains $|H(\theta, z) - H(\theta', z)| \leq 6N\eta(x)|\theta-\theta'|$, where 
\[
\eta(z) =\hat{\eta}+ \left(1+|y|+\sum_{i = 1}^N|x_i|\right)\left(1 +\sum_{i \neq m}^N(|x_i|+|x_m|)\right).
\] 
Similarly, for any $m = 1, \dots, N$, $\theta \in \R^N, z, z' \in \R^{N+1}$, one obtains
\begin{align*}
|H_m(\theta, z) - H_m(\theta, z')|
&\leq 2\left| \left(y -  \sum_{i = 1}^N g_i(\theta) x_{i}\right) g_m(\theta)\sum_{i \neq m}^N g_i(\theta)(x_i- x_m)\right.\\
&\qquad - \left.  \left(y' -  \sum_{i = 1}^N g_i(\theta) x'_{i}\right) g_m(\theta )\sum_{i \neq m}^N g_i(\theta )(x_i'- x_m') \right|\\ 
&\leq 4(N+1)\left( |y'|+\sum_{i = 1}^N|x_i'| +\sum_{i \neq m}^N(|x_i|+|x_m|)\right) |z-z'|\\
&\leq 4(N+1)(\eta(z)+\eta(z'))|z-z'|,
\end{align*}
which further implies $|H(\theta, z) - H(\theta , z')| \leq 4\sqrt{N}(N+1)(\eta(z)+\eta(z'))|z-z'|$.
\end{proof}

\section{Proofs of the results in Section \ref{resultcomparison} and \ref{proofoverview}}\label{proof:sec24}
\begin{proof}[\textbf{Proof of Remark \ref{rem:BoundsOnH}}]\label{proof:rem:BoundsOnH} 
To prove \eqref{mulyan}, one notices that by using Assumptions \ref{iid} and \ref{loclip}
\begin{align*}
| h(\theta)-h(\theta')|		
								& \leq  \mathbb{E}[|H(\theta, X_0) - H(\theta', X_0)|] \leq L_1  \mathbb{E}[\eta(X_0)]|\theta-\theta'|.
\end{align*}
Then, to prove \eqref{growthsup}, one calculates by using Assumption \ref{loclip}
\begin{align*}
|H(\theta,x)| & \leq |H(\theta,x) - H(0,x)|+|H(0,x) - H(0,0)|+|H(0,0)|\\
				& \leq L_1\eta(x)|\theta|+L_2(\eta(x)+\eta(0))|x|+|H(0,0)|\\
				&\leq L_1\eta(x)|\theta|+L_2\bar{\eta}(x)+H_{\star},
\end{align*}
where $\bar{\eta}(x) := (\eta(x)+\eta(0))|x|$, and $H_\star:=|H(0,0)|$. 
\end{proof}

\begin{proof}[\textbf{Proof of Lemma \ref{lem:moment_SGLD_2p}}]\label{proof:lem:moment_SGLD_2p} 
For any $n \in \mathbb{N}$ and $t\in (n, n+1]$, define $\Delta_{n,t}:= \bar{\theta}^{\lambda}_n - \lambda H(\bar{\theta}^{\lambda}_n, X_{n+1})(t-n)$. By using \eqref{SGLDprocess}, it is easily seen that for $t\in (n, n+1]$
\begin{equation*}
\mathbb{E}\left[|\bar{\theta}^{\lambda}_t|^2\left| \bar{\theta}^{\lambda}_n \right.\right] = \mathbb{E}\left[|\Delta_{n,t}|^2\left| \bar{\theta}^{\lambda}_n \right.\right]  + (2 \lambda/\beta)d (t-n).
\end{equation*}
Then, by using Assumptions \ref{iid}, \ref{loclip}, \ref{assum:dissipativity} and Remark \ref{rem:BoundsOnH}, one obtains
\begin{align*}
\mathbb{E}\left[|\Delta_{n,t}|^2\left| \bar{\theta}^{\lambda}_n\right.\right]
& = | \bar{\theta}^{\lambda}_n|^2 - 2\lambda (t-n)\mathbb{E}\left[ \left\langle  \bar{\theta}^{\lambda}_n, H(\bar{\theta}^{\lambda}_n, X_{n+1}) \right\rangle\left| \bar{\theta}^{\lambda}_n\right.\right]\\
&\quad +\lambda^2(t-n)^2\mathbb{E}\left[|H(\bar{\theta}^{\lambda}_n, X_{n+1})|^2\left| \bar{\theta}^{\lambda}_n\right.\right] \\
& \leq | \bar{\theta}^{\lambda}_n|^2 - 2\lambda (t-n)\left\langle  \bar{\theta}^{\lambda}_n, \mathbb{E}\left[ A(X_0)\right] \bar{\theta}^{\lambda}_n\right\rangle +2\lambda (t-n) b\\
&\quad +\lambda^2(t-n)^2\mathbb{E}\left[|L_1\eta(X_{n+1})|\bar{\theta}^{\lambda}_n| +L_2\bar{\eta}(X_{n+1}) +H_\star|^2\left| \bar{\theta}^{\lambda}_n\right.\right]\\
& \leq (1-2a\lambda (t-n))| \bar{\theta}^{\lambda}_n|^2  + 2\lambda^2(t-n)^2L_1^2\mathbb{E}\left[\eta^2(X_0)\right]|\bar{\theta}^{\lambda}_n|^2\\
&\quad +4\lambda^2(t-n)^2L_2^2\mathbb{E}\left[\bar{\eta}^2(X_0)\right] +4\lambda^2(t-n)^2H_\star^2 +2\lambda (t-n) b,
\end{align*}
where the last inequality is obtained by using $(a+b)^2 \leq 2a^2 +2b^2$, for $a,b \geq 0$ twice. For $\lambda <\lambda_{\max}<a/(2 L_1^2\mathbb{E}\left[\eta^2(X_0)\right])$,
\[
\mathbb{E}\left[|\Delta_{n,t}|^2\left|  \bar{\theta}^{\lambda}_n \right.\right]  \leq  (1-a\lambda (t-n))| \bar{\theta}^{\lambda}_n|^2 +\lambda (t-n) c_0,
\]
where $c_0 := 4\lambda_{\max}L_2^2\mathbb{E}\left[\bar{\eta}^2(X_0)\right] +4\lambda_{\max}H_\star^2 +2 b$. Therefore, one obtains
\[
\mathbb{E}\left[| \bar{\theta}^{\lambda}_t|^2\left|  \bar{\theta}^{\lambda}_n \right.\right] \leq (1-a\lambda (t-n))| \bar{\theta}^{\lambda}_n|^2 +\lambda (t-n)c_1,
\]
where $c_1:= c_0+ 2d/\beta$ and the result follows by induction. To calculate a higher moment, denote by $U_{n,t}^{\lambda}:= \{ 2 \lambda \beta^{-1} \}^{1/2}(\tilde{B}_t^{\lambda}-\tilde{B}_n^{\lambda})$, for $t \in (n, n+1]$, one calculates
\begin{align}\label{4thmoment}
\mathbb{E}\left[|  \bar{\theta}^{\lambda}_t|^4\left| \bar{\theta}^{\lambda}_n\right.\right]
& = \mathbb{E}\left[\left(|\Delta_{n,t}|^2+|U_{n,t}^{\lambda}|^2+2\left\langle \Delta_{n,t}, U_{n,t}^{\lambda}\right\rangle\right)^2\left| \bar{\theta}^{\lambda}_n\right.\right]\nonumber\\
& = \mathbb{E}\left[ |\Delta_{n,t}|^4+|U_{n,t}^{\lambda}|^4+2|\Delta_{n,t}|^2|U_{n,t}^{\lambda}|^2+4|\Delta_{n,t}|^2\left\langle\Delta_{n,t}, U_{n,t}^{\lambda}\right\rangle \right.\nonumber\\
&\qquad \left. +4|U_{n,t}^{\lambda}|^2\left\langle \Delta_{n,t}, U_{n,t}^{\lambda}\right\rangle+4\left(\left\langle \Delta_{n,t}, U_{n,t}^{\lambda}\right\rangle\right)^2 \left| \bar{\theta}^{\lambda}_n\right.\right]\nonumber\\
& \leq \mathbb{E}\left[ |\Delta_{n,t}|^4+|U_{n,t}^{\lambda}|^4+6|\Delta_{n,t}|^2|U_{n,t}^{\lambda}|^2 \left| \bar{\theta}^{\lambda}_n\right.\right]\nonumber\\
&\leq (1+a\lambda(t-n))\mathbb{E}\left[ |\Delta_{n,t}|^4\left| \bar{\theta}^{\lambda}_n\right.\right] + (1+9/(a\lambda(t-n))) \mathbb{E}\left[ |U_{n,t}^{\lambda}|^4\right].
\end{align}
where the last inequality holds due to $2uv \leq \dot{\epsilon} u^2 +\dot{\epsilon}^{-1}v^2$, for $u,v \geq 0$ and $\dot{\epsilon}>0$ with $u = |\Delta_{n,t}|^2, v = 3|U_{n,t}^{\lambda}|^2, \dot{\epsilon} = a\lambda(t-n)$, and due to the independence of $U_{n,t}^{\lambda}$ and $\bar{\theta}^{\lambda}_n$. Then, by using the Cauchy-Schwarz inequality, one obtains
\begin{align*}
&\mathbb{E}\left[|\Delta_{n,t}|^4\left| \bar{\theta}^{\lambda}_n \right.\right]\\
&= \mathbb{E}\left[\left(| \bar{\theta}^{\lambda}_n|^2 -2\lambda (t-n) \left\langle  \bar{\theta}^{\lambda}_n, H(\bar{\theta}^{\lambda}_n, X_{n+1}) \right\rangle+\lambda^2(t-n)^2|H(\bar{\theta}^{\lambda}_n, X_{n+1})|^2\right)^2\left| \bar{\theta}^{\lambda}_n \right.\right] \\
& =  \mathbb{E}\left[ | \bar{\theta}^{\lambda}_n|^4+ 4\lambda^2 (t-n)^2\left( \left\langle  \bar{\theta}^{\lambda}_n, H(\bar{\theta}^{\lambda}_n, X_{n+1}) \right\rangle \right)^2 \right.\\
&\quad + \lambda^4(t-n)^4|H(\bar{\theta}^{\lambda}_n, X_{n+1})|^4 - 4\lambda (t-n) \left\langle  \bar{\theta}^{\lambda}_n, H(\bar{\theta}^{\lambda}_n, X_{n+1}) \right\rangle | \bar{\theta}^{\lambda}_n|^2\\
&\quad +2\lambda^2(t-n)^2| \bar{\theta}^{\lambda}_n|^2|H(\bar{\theta}^{\lambda}_n, X_{n+1})|^2 \\
&\quad \left.  - 4\lambda^3 (t-n)^3 \left\langle  \bar{\theta}^{\lambda}_n, H(\bar{\theta}^{\lambda}_n, X_{n+1}) \right\rangle |H(\bar{\theta}^{\lambda}_n, X_{n+1})|^2 \left| \bar{\theta}^{\lambda}_n \right.\right] \\
&\leq | \bar{\theta}^{\lambda}_n|^4+\mathbb{E}\left[6\lambda^2(t-n)^2| \bar{\theta}^{\lambda}_n|^2|H(\bar{\theta}^{\lambda}_n, X_{n+1})|^2 - 4\lambda (t-n) \left\langle  \bar{\theta}^{\lambda}_n, H(\bar{\theta}^{\lambda}_n, X_{n+1}) \right\rangle | \bar{\theta}^{\lambda}_n|^2 \right. \\
&\quad \left. -4\lambda^3 (t-n)^3|H(\bar{\theta}^{\lambda}_n, X_{n+1})|^2 \left\langle  \bar{\theta}^{\lambda}_n, H(\bar{\theta}^{\lambda}_n, X_{n+1}) \right\rangle +\lambda^4(t-n)^4|H(\bar{\theta}^{\lambda}_n, X_{n+1})|^4   \left| \bar{\theta}^{\lambda}_n \right.\right].
\end{align*}
By Remark \ref{rem:BoundsOnH}, the independence of $X_{n+1}$ and $\bar{\theta}^{\lambda}_n$, and by using $(u+v+\nu)^s  \leq 2^{s-1}(u+v)^s+2^{s-1}\nu^s \leq 2^{2s-2}(u^s+v^s)+2^{s-1}\nu^s$ for $u,v, \nu \geq 0, s \geq 1$, it follows that, for $q\geq 1$,
\begin{align}\label{Hest}
\begin{split}
\mathbb{E}\left[|H(\bar{\theta}^{\lambda}_n, X_{n+1})|^q	\left| \bar{\theta}^{\lambda}_n \right.\right]
&\leq 2^{q-1}L_1^q\mathbb{E}\left[\eta^q(X_0)\right]|\bar{\theta}^{\lambda}_n|^q +2^{2q-2}L_2^q\mathbb{E}\left[|\bar{\eta}^q(X_0)\right] +2^{2q-2}H_\star^q.
\end{split}
\end{align}
By using Assumption \ref{assum:dissipativity} and by taking $q = 2, 3, 4$ in \eqref{Hest}, one obtains
\begin{align*}
\mathbb{E}\left[|\Delta_{n,t}|^4\left| \bar{\theta}^{\lambda}_n \right.\right]
&\leq  (1- 4a\lambda (t-n) )| \bar{\theta}^{\lambda}_n|^4 +4b\lambda (t-n) | \bar{\theta}^{\lambda}_n|^2 \\
&\quad +12\lambda^2(t-n)^2L_1^2\mathbb{E}\left[\eta^2(X_0)\right]| \bar{\theta}^{\lambda}_n|^4 + 16\lambda^3(t-n)^3L_1^3\mathbb{E}\left[\eta^3(X_0)\right]| \bar{\theta}^{\lambda}_n|^4\\
&\quad +8\lambda^4(t-n)^4L_1^4\mathbb{E}\left[\eta^4(X_0)\right]| \bar{\theta}^{\lambda}_n|^4\\
&\quad +24\lambda^2(t-n)^2\left(L_2^2\mathbb{E}\left[\bar{\eta}^2(X_0)\right]+H_\star^2\right)| \bar{\theta}^{\lambda}_n|^2\\
&\quad +64\lambda^3(t-n)^3\left(L_2^3\mathbb{E}\left[\bar{\eta}^3(X_0)\right]+H_\star^3\right)| \bar{\theta}^{\lambda}_n|\\
&\quad +64\lambda^4(t-n)^4\left(L_2^4\mathbb{E}\left[\bar{\eta}^4(X_0)\right]+H_\star^4\right)
\end{align*}
which implies, by using $\lambda < \lambda_{\max}$
\begin{align}\label{ubM}
\begin{split}
\mathbb{E}\left[|\Delta_{n,t}|^4\left| \bar{\theta}^{\lambda}_n \right.\right]
&\leq  (1- 3a\lambda (t-n) )| \bar{\theta}^{\lambda}_n|^4 +4b\lambda (t-n)| \bar{\theta}^{\lambda}_n|^2  \\
&\quad +24\lambda^2(t-n)^2\left(L_2^2\mathbb{E}\left[\bar{\eta}^2(X_0)\right]+H_\star^2\right)| \bar{\theta}^{\lambda}_n|^2\\
&\quad +64\lambda^3(t-n)^3\left(L_2^3\mathbb{E}\left[\bar{\eta}^3(X_0)\right]+H_\star^3\right)| \bar{\theta}^{\lambda}_n|\\
&\quad +64\lambda^4(t-n)^4\left(L_2^4\mathbb{E}\left[\bar{\eta}^4(X_0)\right]+H_\star^4\right).
\end{split}
\end{align}
For $|\bar{\theta}^{\lambda}_n|>(8ba^{-1}+48a^{-1}\lambda_{\max}(L_2^2\mathbb{E}\left[\bar{\eta}^2(X_0)\right]+H_\star^2))^{1/2}$, we have
\begin{equation}\label{ubM1}
- \frac{a\lambda (t-n) }{2}| \bar{\theta}^{\lambda}_n|^4 +4b\lambda (t-n)| \bar{\theta}^{\lambda}_n|^2+24\lambda^2(t-n)^2\left(L_2^2\mathbb{E}\left[\bar{\eta}^2(X_0)\right]+H_\star^2\right)| \bar{\theta}^{\lambda}_n|^2<0.
\end{equation}
Similarly, for $| \bar{\theta}^{\lambda}_n|>(128a^{-1}\lambda_{\max}^2(L_2^3\mathbb{E}\left[\bar{\eta}^3(X_0)\right]+H_\star^3))^{1/3}$
\begin{equation}\label{ubM2}
- \frac{a\lambda (t-n) }{2}| \bar{\theta}^{\lambda}_n|^4 +64\lambda^3(t-n)^3\left(L_2^3\mathbb{E}\left[\bar{\eta}^3(X_0)\right]+H_\star^3\right)| \bar{\theta}^{\lambda}_n|<0.
\end{equation}
Denote by
\begin{align}\label{estM}
\begin{split}
M &:= \max\{(8ba^{-1}+48a^{-1}\lambda_{\max}(L_2^2\mathbb{E}\left[\bar{\eta}^2(X_0)\right]+H_\star^2))^{1/2},\\
	&\hspace{4em} (128a^{-1}\lambda_{\max}^2(L_2^3\mathbb{E}\left[\bar{\eta}^3(X_0)\right]+H_\star^3))^{1/3}\}.
\end{split}
\end{align}
Moreover, denote by $\mathsf{A}_{n,M} := \{\omega \in \Omega: |\bar{\theta}^{\lambda}_n(\omega)| >M \}$. Then, by substituting \eqref{ubM1}, \eqref{ubM2} into \eqref{ubM}, one obtains, 
\begin{align*}
\mathbb{E}\left[|\Delta_{n,t}|^4\mathbbm{1}_{\mathsf{A}_{n,M}}\left| \bar{\theta}^{\lambda}_n \right.\right]
& \leq  (1- 2a\lambda (t-n) )| \bar{\theta}^{\lambda}_n|^4  \mathbbm{1}_{\mathsf{A}_{n,M}} \\
&\quad +64\lambda^4(t-n)^4\left(L_2^4\mathbb{E}\left[\bar{\eta}^4(X_0)\right]+H_\star^4\right)\mathbbm{1}_{\mathsf{A}_{n,M}}.
\end{align*}
Similarly, we have
\begin{align*}
\mathbb{E}\left[|\Delta_{n,t}|^4\mathbbm{1}_{\mathsf{A}_{n,M}^{\mathsf{c}}} \left| \bar{\theta}^{\lambda}_n \right.\right]
&\leq  (1- 2a\lambda (t-n) )| \bar{\theta}^{\lambda}_n|^4 \mathbbm{1}_{\mathsf{A}_{n,M}^{\mathsf{c}}}+4b\lambda (t-n)M^2\mathbbm{1}_{\mathsf{A}_{n,M}^{\mathsf{c}}}  \\
&\quad +24\lambda^2(t-n)^2\left(L_2^2\mathbb{E}\left[\bar{\eta}^2(X_0)\right]+H_\star^2\right)M^2\mathbbm{1}_{\mathsf{A}_{n,M}^{\mathsf{c}}}\\
&\quad +64\lambda^3(t-n)^3\left(L_2^3\mathbb{E}\left[\bar{\eta}^3(X_0)\right]+H_\star^3\right)M\mathbbm{1}_{\mathsf{A}_{n,M}^{\mathsf{c}}}\\
&\quad +64\lambda^4(t-n)^4\left(L_2^4\mathbb{E}\left[\bar{\eta}^4(X_0)\right]+H_\star^4\right)\mathbbm{1}_{\mathsf{A}_{n,M}^{\mathsf{c}}}.
\end{align*}
Combining the two cases yields
\begin{equation}\label{deltaest4}
\mathbb{E}\left[|\Delta_{n,t}|^4\left| \bar{\theta}^{\lambda}_n \right.\right]\leq  (1- 2a\lambda (t-n) )| \bar{\theta}^{\lambda}_n|^4 + \lambda (t-n) c_2,
\end{equation}
where $
c_2 := 4bM^2 + 152(1+\lambda_{\max})^3 \left((1+L_2)^4\mathbb{E}\left[(1+\bar{\eta}(X_0))^4\right]+(1+H_\star)^4\right)(1+M)^2  
$
 with $M$ given in \eqref{estM}. Substituting \eqref{deltaest4} into \eqref{4thmoment}, one obtains
\begin{align*}
\mathbb{E}\left[|\bar{\theta}^{\lambda}_t|^4\left| \bar{\theta}^{\lambda}_n \right.\right]
& \leq  (1+a\lambda(t-n))(1- 2a\lambda (t-n) )| \bar{\theta}^{\lambda}_n|^4 \\
&\quad +(1+a\lambda(t-n)) \lambda (t-n) c_2 +12d^2\lambda^2 \beta^{-2}(t-n)^2 (1+9/(a\lambda(t-n)))\\
& \leq (1-  a\lambda (t-n) )| \bar{\theta}^{\lambda}_n|^4 +\lambda (t-n) c_3,
\end{align*}
where $c_3 	:=  (1+a\lambda_{\max})c_2+12d^2\beta^{-2}(\lambda_{\max}+9a^{-1})$.
Finally, for any $n \in \mathbb{N}, t \in (n, n+1]$, $0<\lambda \leq \lambda_{\max} $, one obtains,
\begin{align*}
\mathbb{E}\left[|\bar{\theta}^{\lambda}_t|^4 \right]
& \leq (1-  a\lambda (t-n) )\mathbb{E}\left[| \bar{\theta}^{\lambda}_n|^4\right] +\lambda (t-n) c_3\\
& \leq (1-a\lambda(t-n) ) (1-a \lambda ) \E\left[|\bar{\theta}^{\lambda}_{n-1}|^4 \right] +\lambda_{\max} c_3 +\lambda c_3\\
& \leq (1- a\lambda(t-n) ) (1-a\lambda  )^2 \E\left[|\bar{\theta}^{\lambda}_{n-2}|^4 \right]\\
&\quad +  \lambda_{\max} c_3 +\lambda c_3(1+(1-a\lambda  ))\\
&\leq \dots \\
&\leq  (1-a\lambda(t-n) )(1-a \lambda)^n \E\left[|\theta_0|^4\right]  +c_3(\lambda_{\max} +1/a),
\end{align*}
which completes the proof.
\end{proof}

\begin{proof}[\textbf{Proof of Lemma~\ref{zetaprocessmoment}}]\label{proof:zetaprocessmoment}
For any $p\geq 1$, application of Ito's lemma and taking expectation yields
\begin{align*}
\E[V_p(\bar{\zeta}_t^{\lambda,n})] = \E[V_p( \bar{\theta}^{\lambda}_{nT})] + \int_{nT}^t \E\left[\lambda \frac{\Delta V_p(\bar{\zeta}_s^{\lambda,n})}{\beta} - \lambda \langle h(\bar{\zeta}_s^{\lambda,n}), \nabla V_p(\bar{\zeta}_s^{\lambda,n}) \rangle\right] \rmd s.
\end{align*}
Differentiating both sides and using Lemma~\ref{lem:PreLimforDriftY}, we arrive at
\begin{align*}
\frac{\rmd}{\rmd t} \E[V_p(\bar{\zeta}_t^{\lambda,n})] 
&= \E\left[ \lambda \frac{\Delta V_p(\bar{\zeta}_t^{\lambda,n})}{\beta} - \lambda \langle h(\bar{\zeta}_t^{\lambda,n}), \nabla V_p(\bar{\zeta}_t^{\lambda,n}) \rangle\right] \\
&\leq -\lambda \bar{c}(p) \E[V_p(\bar{\zeta}_t^{\lambda,n})] + \lambda \tilde{c}(p),
\end{align*}
which yields
\begin{align*}
\E[V_p(\bar{\zeta}_t^{\lambda,n})] &\leq e^{-\lambda (t - nT) \bar{c}(p)} \E[V_p( \bar{\theta}^{\lambda}_{nT})] + \tilde{c}(p)/\bar{c}(p) \left( 1 - e^{-\lambda \bar{c}(p) (t - nT)} \right) \\
&\leq e^{-\lambda (t - nT) \bar{c}(p)} \E[V_p(\bar{\theta}^{\lambda}_{nT})] + \tilde{c}(p)/\bar{c}(p).
\end{align*}
Now for $p = 2$, by using Lemma \ref{lem:moment_SGLD_2p}, Corollary~\ref{corr:Moments} and Lemma \ref{lem:PreLimforDriftY}, we obtain
\begin{align*}
\E[V_2(\bar{\zeta}_t^{\lambda,n})] &\leq e^{-\lambda(t-nT) \bar{c}(2)} \E[V_2(\bar{\theta}^{\lambda}_{nT})] +\tilde{c}(2)/\bar{c}(2)\\
&\leq (1-a\lambda)^{nT} e^{-\lambda(t-nT) \bar{c}(2)} \mathbb{E}[ V_2(\theta_0)]+ \tilde{c}(2)/\bar{c}(2) +c_1 (\lambda_{\max}+a^{-1})+1\\
&\leq e^{-a\lambda  t/2}   \E[V_2(\theta_0)] +3 \mathrm{v}_2(\overline{M}_2)+c_1(\lambda_{\max}+a^{-1})+1,
\end{align*}
where the last inequality holds due to $0 \leq 1-z \leq e^{-z}$ for $z \geq 0$ and $\bar{c}(2) =a/2$. Similarly, for $p = 4$, one obtains
\begin{align*}
\E[V_4(\bar{\zeta}_t^{\lambda,n})] &\leq e^{-\lambda(t-nT) \bar{c}(4)} \E[V_4(\bar{\theta}^{\lambda}_{nT})] + \tilde{c}(4)/\bar{c}(4)\\
&\leq 2(1 - a\lambda)^{nT} e^{-\lambda(t-nT) \bar{c}(4)} \E[V_4(\theta_0)] + \tilde{c}(4)/\bar{c}(4)+2c_3 (\lambda_{\max}+a^{-1})+2\\
&\leq 2e^{-a\lambda  t}   \E[V_4(\theta_0)] +3 \mathrm{v}_4(\overline{M}_4)+2c_3 (\lambda_{\max}+a^{-1})+2,
\end{align*}
where the last inequality holds due to $0 \leq 1-z \leq e^{-z}$ for $z \geq 0$ and $\bar{c}(4) =a$.
\end{proof}

\begin{proof}[\textbf{Proof of Lemma~\ref{convergencepart1}}]\label{proof:convergencepart1}

To handle the first term in \eqref{decomposition}, we start by establishing an upper bound in Wasserstein-2 distance and the statment follows by noticing $W_1\leq W_2$. By employing synchronous coupling, using \eqref{SGLDprocess} and the definition of $\bar{\zeta}_t^{\lambda,n} $ in Definition \ref{zetaprocess}, one obtains, for any $t  \in (nT, (n+1)T]$,
\begin{align*}
\left| \bar{\zeta}_t^{\lambda,n} - \bar{\theta}^{\lambda}_t\right|
& \leq \lambda \left| \int_{nT}^t \left[ H(\bar{\theta}^{\lambda}_\floor{s}, X_{\ceil{s}}) - h(\bar{\zeta}_s^{\lambda,n}) \right] \rmd s \right|\\
&\leq \lambda \left| \int_{nT}^t \left[ H(\bar{\theta}^{\lambda}_\floor{s},X_{\ceil{s}}) - H(\bar{\zeta}_s^{\lambda,n},X_{\ceil{s}})\right] \rmd s \right| \\
&\quad+ \lambda \left| \int_{nT}^t \left[ h(\bar{\zeta}_s^{\lambda,n}) - H(\bar{\zeta}_s^{\lambda,n},X_{\ceil{s}})\right] \rmd s \right|\\
&\leq \lambda L_1 \int_{nT}^t \eta(X_\ceil{s}) \left| \bar{\theta}^{\lambda}_\floor{s} - \bar{\zeta}_s^{\lambda,n}\right| \rmd s   + \lambda \left| \int_{nT}^t \left[ h(\bar{\zeta}_s^{\lambda,n}) - H(\bar{\zeta}_s^{\lambda,n},X_{\ceil{s}})\right] \rmd s \right|,
\end{align*}
where the last inequality holds due to Assumption~\ref{loclip}.
Now taking squares of both sides, using $(a+b)^2 \leq 2a^2 + 2 b^2$ for $a, b >0$, and then taking expectations lead to
\begin{align*}
\E\left[\left| \bar{\zeta}_t^{\lambda,n} -\bar{\theta}^{\lambda}_t \right|^2\right]
&\leq 2 \lambda L_1^2 \int_{nT}^t \E\left[\eta^2(X_0)\right] \E\left[\left| \bar{\theta}^{\lambda}_\floor{s} - \bar{\zeta}_s^{\lambda,n} \right|^2\right] \rmd s \\
&\quad + 2 \lambda^2 \E\left[\left| \int_{nT}^t \left[ h(\bar{\zeta}_s^{\lambda,n}) - H(\bar{\zeta}_s^{\lambda,n},X_{\ceil{s}})\right] \rmd s \right|^2\right].
\end{align*}
where 
the expectation splits over terms in the first integral due to the independence of $X_{\ceil{s}}$ from the rest of the random variables. Using $\lambda T \leq 1$, Lemma \ref{onestepest} and $(a+b)^2 \leq 2 a^2 + 2 b^2$ once again, we obtain

\begin{align}\label{proof:ErrorSplit}
\E\left[\left| \bar{\zeta}_t^{\lambda,n} - \bar{\theta}^{\lambda}_t \right|^2\right]
&\leq 4 \lambda L_1^2 \E\left[\eta^2(X_0)\right] \int_{nT}^t \E\left[\left| \bar{\theta}^{\lambda}_\floor{s} -  \bar{\theta}^{\lambda}_s \right|^2\right] \rmd s \nonumber\\
&\quad + 4 \lambda  L_1^2 \E\left[\eta^2(X_0)\right] \int_{nT}^t \E\left[\left| \bar{\theta}^{\lambda}_s - \bar{\zeta}_s^{\lambda,n} \right|^2\right] \rmd s \nonumber\\
&\quad+ 2 \lambda^2 \E\left[\left| \int_{nT}^t \left[ h(\bar{\zeta}_s^{\lambda,n}) - H(\bar{\zeta}_s^{\lambda,n},X_{\ceil{s}})\right] \rmd s \right|^2\right] \nonumber\\
&\leq 4 \lambda L_1^2  \E\left[\eta^2(X_0)\right]  (e^{-a\lambda nT}\bar{\sigma}_Y \mathbb{E}[V_2(\theta_0)] + \tilde{\sigma}_Y) \nonumber\\
&\quad + 4 \lambda L_1^2 \E\left[\eta^2(X_0)\right] \int_{nT}^t \E\left[\left|  \bar{\theta}^{\lambda}_s - \bar{\zeta}_s^{\lambda,n} \right|^2\right] \rmd s \nonumber\\
&\quad+ 2 \lambda^2 \E\left[\left| \int_{nT}^t \left[ h(\bar{\zeta}_s^{\lambda,n}) - H(\bar{\zeta}_s^{\lambda,n},X_{\ceil{s}})\right] \rmd s \right|^2\right].
\end{align}
where $\bar{\sigma}_Y$ and $\tilde{\sigma}_Y$ are provided in \eqref{sigmaY}. Next, we bound the last term in \eqref{proof:ErrorSplit} by partitioning the last integral. Assume that $nT + K < t \leq nT + K + 1$ where $K + 1 \leq T, K \in \mathbb{N}$. Thus we can write
\begin{align*}
\left| \int_{nT}^t \left[ h(\bar{\zeta}_s^{\lambda,n}) - H(\bar{\zeta}_s^{\lambda,n},X_{\ceil{s}})\right] \rmd s \right| = \left| \sum_{k=1}^{K} I_k + R_K \right|
\end{align*}
where $I_k  := \int_{nT + (k-1)}^{nT + k} [h(\bar{\zeta}_s^{\lambda,n}) - H(\bar{\zeta}_s^{\lambda,n},X_{nT+k})] \rmd s$, and $R_K := \int_{nT+K}^{t} [h(\bar{\zeta}_s^{\lambda,n}) - H(\bar{\zeta}_s^{\lambda,n},X_{nT+K + 1})] \rmd s$.
Taking squares of both sides
\begin{align*}
\left| \sum_{k=1}^{K} I_k + R_K \right|^2 = \sum_{k=1}^K | I_k |^2 + 2 \sum_{k=2}^K \sum_{j = 1}^{k-1} \langle I_k, I_j \rangle +2 \sum_{k=1}^K \langle I_k, R_K \rangle + |R_K|^2,
\end{align*}
Finally, we take expectations of both sides. Define the filtration $\mathcal{H}_t = \mathcal{F}^{\lambda}_{\infty} \vee \mathcal{G}_{\lfloor t \rfloor}$. We first note that for any $k =2, \dots, K$, $j = 1, \dots, k-1$,
\begin{align*}
\E \left[\langle I_k, I_j \rangle\right]
 &= \E\left[ \E [\langle I_k, I_j \rangle | \mathcal{H}_{nT+j} ] \right], \\
&= \E\left[ \E \left[\left\langle \int_{nT + (k-1)}^{nT + k} [H(\bar{\zeta}_s^{\lambda,n},X_{nT + k}) - h(\bar{\zeta}_s^{\lambda,n})] \rmd s, \right.\right.\right.\\
&\hspace{5em} \left.\left.\left.\left.\int_{nT + (j-1)}^{nT + j} [H(\bar{\zeta}_s^{\lambda,n},X_{nT + j}) - h(\bar{\zeta}_s^{\lambda,n})] \rmd s \right\rangle \right|  \mathcal{H}_{nT+j} \right] \right], \\
& = \E\left[ \left\langle \int_{nT + (k-1)}^{nT + k} \E \left[\left. H(\bar{\zeta}_s^{\lambda,n},X_{nT + k}) - h(\bar{\zeta}_s^{\lambda,n})\right|  \mathcal{H}_{nT+j} \right]\rmd s, \right.\right. \\
& \hspace{5em} \left.\left.  \int_{nT + (j-1)}^{nT + j} [H(\bar{\zeta}_s^{\lambda,n},X_{nT + j}) - h(\bar{\zeta}_s^{\lambda,n})] \rmd s \right\rangle  \right] = 0.
\end{align*}
By the same argument $\E \langle I_k, R_K\rangle = 0$ for all $1 \leq k \leq K$. Therefore, the last term of \eqref{proof:ErrorSplit} is bounded as
\begin{align*}
2 \lambda^2 \E\left[ \left| \int_{nT}^t \left[ h(\bar{\zeta}_s^{\lambda,n}) - H(\bar{\zeta}_s^{\lambda,n},X_{\ceil{s}})\right] \rmd s \right|^2\right] &= 2 \lambda^2 \sum_{k=1}^K \E\left[ |I_k|^2\right] + 2 \lambda^2 \E\left[ |R_K|^2\right] \\
&\leq 4e^{-a\lambda nT/2} \lambda (\bar{\sigma}_Z\mathbb{E}[V_2(\theta_0)]+\tilde{\sigma}_Z),
\end{align*}
where the last inequality holds due to Lemma~\ref{lem:boundedVariance} and $\bar{\sigma}_Z$ and $\tilde{\sigma}_Z$ are provided in \eqref{sigmaZ}. Therefore, the bound \eqref{proof:ErrorSplit} becomes
\begin{align*}
\E\left[\left| \bar{\zeta}_t^{\lambda,n} - \bar{\theta}^{\lambda}_t \right|^2\right]
&\leq  4 \lambda  L_1^2\E\left[\eta^2(X_0)\right]\int_{nT}^t \E\left[\left|  \bar{\theta}^{\lambda}_s - \bar{\zeta}_s^{\lambda,n} \right|^2\right] \rmd s\\
&\quad +4e^{-a\lambda nT/2} \lambda (L_1^2 \E\left[\eta^2(X_0)\right]  \bar{\sigma}_Y + \bar{\sigma}_Z )\mathbb{E}[V_2(\theta_0)] \\
&\quad +4 \lambda ( L_1^2 \E\left[\eta^2(X_0)\right]\tilde{\sigma}_Y+\tilde{\sigma}_Z).
\end{align*}
Using Gr\"onwall's inequality leads
\begin{align*}
\E\left[\left| \bar{\zeta}_t^{\lambda,n} - \bar{\theta}^{\lambda}_t \right|^2\right]
&\leq \lambda e^{4  L_1^2 \E\left[\eta^2(X_0)\right]}\left[ 4 e^{-a\lambda nT/2} (L_1^2 \E\left[\eta^2(X_0)\right] \bar{\sigma}_Y +\bar{\sigma}_Z )\mathbb{E}[V_2(\theta_0)] \right.\\
&\hspace{8em} \left. +4   ( L_1^2 \E\left[\eta^2(X_0)\right] \tilde{\sigma}_Y+\tilde{\sigma}_Z)\right].
\end{align*}
which implies by $\lambda T \geq 1/2$,
\begin{equation}\label{vibd1}
W^2_2(\calL(\bar{\theta}^{\lambda}_t),\calL(\bar{\zeta}_t^{\lambda,n})) \leq \E\left| \bar{\zeta}_t^{\lambda,n} - \bar{\theta}^{\lambda}_t \right|^2  \leq  \lambda (e^{-an/4}\bar{C}_{2,1}\mathbb{E}[V_2(\theta_0)]  +\bar{C}_{2,2}) ,
\end{equation}
where
\begin{align}\label{barc2}
\begin{split}
&\bar{C}_{2,1} := 4 e^{4  L_1^2 \E\left[\eta^2(X_0)\right]}(L_1^2 \E\left[\eta^2(X_0)\right] \bar{\sigma}_Y +\bar{\sigma}_Z ), \\
& \bar{C}_{2,2} := 4  e^{4  L_1^2 \E\left[\eta^2(X_0)\right]} ( L_1^2 \E\left[\eta^2(X_0)\right] \tilde{\sigma}_Y+\tilde{\sigma}_Z)
\end{split}
\end{align}
with $ \bar{\sigma}_Y$, $ \tilde{\sigma}_Y$ provided in \eqref{sigmaY} and $\bar{\sigma}_Z$, $\tilde{\sigma}_Z$ given in \eqref{sigmaZ}.
\end{proof}

\begin{proof}[\textbf{Proof of Lemma~\ref{convergencepart2}}]\label{proof:convergencepart2}

To upper bound the second term $W_1(\mathcal{L}(\bar{\zeta}_t^{\lambda,n}),\calL(Z_t^\lambda))$ in \eqref{decomposition}, we adapt the proof from Lemma~3.18 in \cite{nonconvex}. Recall the definition of $w_{1,2}$ given in \eqref{w1p}, and the fact that $W_1(\mu, \nu) \leq w_{1,2}(\mu, \nu)$ for any $ \mu,\nu \in \mathcal{P}_{\, V_2}$. By Proposition \ref{contr}, one obtains, for any $t  \in (nT, (n+1)T]$,
\begin{align*}
W_1(\mathcal{L}(\bar{\zeta}_t^{\lambda,n}),\calL(Z_t^\lambda))  
&\leq \sum_{k=1}^n W_1(\calL(\bar{\zeta}_t^{\lambda,k}),\calL(\bar{\zeta}_t^{\lambda,k-1}))  \nonumber \\
&\leq \sum_{k=1}^n w_{1,2}(\calL(\zeta^{kT,\bar{\theta}^{\lambda}_{kT}, \lambda}_t ),\calL(\zeta^{kT,\bar{\zeta}_{kT}^{\lambda,k-1}, \lambda}_t)) \nonumber\\
&\leq \hat{c} \sum_{k=1}^n \exp(-\dot{c} (n-k)) w_{1,2}(\calL(\bar{\theta}^{\lambda}_{kT}),\calL(\bar{\zeta}_{kT}^{\lambda,k-1})), \nonumber 
\end{align*}
which implies, by using Cauchy-Schwarz inequality, Young's inequality, Lemma~\ref{convergencepart1}, Corollary \ref{corr:Moments} and Lemma \ref{zetaprocessmoment},
\begin{align*}
&W_1(\mathcal{L}(\bar{\zeta}_t^{\lambda,n}),\calL(Z_t^\lambda)) \\
&\leq \hat{c} \sum_{k=1}^n \exp(-\dot{c} (n-k))  W_2(\calL(\bar{\theta}^{\lambda}_{kT}),\calL(\bar{\zeta}_{kT}^{\lambda,k-1}))  \left[1 + \left\lbrace \E[V_4(\bar{\theta}^{\lambda}_{kT})]\right\rbrace^{1/2} \right.\\
&\hspace{18em} \left.+ \left\lbrace\E[V_4(\bar{\zeta}_{kT}^{\lambda,k-1})] \right\rbrace^{1/2}\right] \nonumber \\
&\leq (\sqrt{\lambda})^{-1}\hat{c} \sum_{k=1}^n \exp(-\dot{c} (n-k))  W^2_2(\calL(\bar{\theta}^{\lambda}_{kT}),\calL(\bar{\zeta}_{kT}^{\lambda,k-1})) \nonumber\\
&\quad +3\sqrt{\lambda}\hat{c} \sum_{k=1}^n \exp(-\dot{c} (n-k))  \left[1 + \E[V_4(\bar{\theta}^{\lambda}_{kT})] +\E[V_4(\bar{\zeta}_{kT}^{\lambda,k-1})] \right] \nonumber\\
& \leq \sqrt{\lambda} \hat{c} \sum_{k=1}^n \exp(-\dot{c} (n-k)) (e^{-a(k-1)/4}\bar{C}_{2,1}\mathbb{E}[V_2(\theta_0)]  +\bar{C}_{2,2}) \nonumber\\
&\quad +3\sqrt{\lambda}\hat{c} \sum_{k=1}^n \exp(-\dot{c} (n-k))  \left[1 + \E[V_4(\bar{\theta}^{\lambda}_{kT})] +\E[V_4(\bar{\zeta}_{kT}^{\lambda,k-1})] \right] \nonumber\\
& \leq  \sqrt{\lambda}e^{-\min\{\dot{c},a/4\}n}n\hat{c} (e^{\min\{\dot{c},a/4\}}\bar{C}_{2,1}\mathbb{E}[V_2(\theta_0)]  +12\mathbb{E}[V_4(\theta_0)] ) \nonumber\\
&\quad + \sqrt{\lambda}\frac{\hat{c} }{1 - \exp(-\dot{c})}(\bar{C}_{2,2}+12c_3(\lambda_{\max}+a^{-1})+9\mathrm{v}_4(\overline{M}_4)+15) \nonumber\\
&\leq \sqrt{\lambda}(e^{-\min\{\dot{c},a/4\}n/2}\bar{C}_{2,3}\mathbb{E}[V_4(\theta_0)] +\bar{C}_{2,4})\nonumber\\
&=\sqrt{\lambda}(e^{-\dot{c}n/2}\bar{C}_{2,3}\mathbb{E}[V_4(\theta_0)] +\bar{C}_{2,4}),
\end{align*}
where the last inequality holds by applying the inequality $e^{-\alpha n}(n+1) \leq 1+\alpha^{-1}$, for $\alpha>0$ with $\alpha = \min\{\dot{c},a/4\}/2$, and the last equality holds by noticing $\min\{\dot{c},a/4\} = \dot{c}$ with $\dot{c}$ given in \eqref{dotc}. The explicit expressions for the constants $\bar{C}_{2,3}, \bar{C}_{2,4}$ are given below:
\begin{align}\label{mainthmc2}
\begin{split}
\bar{C}_{2,3}&: =\hat{c}\left(1+\frac{2}{\dot{c}}\right) (e^{a/4}\bar{C}_{2,1}  +12)\\
\bar{C}_{2,4}& := \frac{\hat{c} }{1 - \exp(-\dot{c})}(\bar{C}_{2,2}+12c_3(\lambda_{\max}+a^{-1})+9\mathrm{v}_4(\overline{M}_4)+15)
\end{split}
\end{align}
with $\bar{C}_{2,1} $, $\bar{C}_{2,2}$ given in \eqref{barc2}, $\hat{c} $, $\dot{c}$ given in Lemma \ref{contractionconst}, $c_3$ given in \eqref{4thmomentconst}, and $\overline{M}_4$ given in Lemma \ref{lem:PreLimforDriftY}.
\end{proof}

\begin{proof}[\textbf{Proof of Corollary~\ref{convergencepart2w2}}]\label{proof:convergencepart2w2}
One notices that $W_2 \leq \sqrt{2w_{1,2}}$, then, by using similar arguments as in the proof of Lemma \ref{convergencepart2}, one obtains
\begin{align*}
&W_2(\mathcal{L}(\bar{\zeta}_t^{\lambda,n}),\calL(Z_t^\lambda))  \nonumber\\
&\leq \sum_{k=1}^n W_2(\calL(\bar{\zeta}_t^{\lambda,k}),\calL(\bar{\zeta}_t^{\lambda,k-1}))  \nonumber \\
&\leq \sum_{k=1}^n \sqrt{2}w^{1/2}_{1,2}(\calL(\zeta^{kT,\bar{\theta}^{\lambda}_{kT}, \lambda}_t ),\calL(\zeta^{kT,\bar{\zeta}_{kT}^{\lambda,k-1}, \lambda}_t)) \nonumber\\
&\leq \sqrt{2\hat{c}}\sum_{k=1}^n \exp(-\dot{c} (n-k)/2)  W^{1/2}_2(\calL(\bar{\theta}^{\lambda}_{kT}),\calL(\bar{\zeta}_{kT}^{\lambda,k-1}))  \left[1 + \left\lbrace \E[V_4(\bar{\theta}^{\lambda}_{kT})]\right\rbrace^{1/2} \right.\\
&\hspace{18em} \left.+ \left\lbrace\E[V_4(\bar{\zeta}_{kT}^{\lambda,k-1})] \right\rbrace^{1/2}\right]^{1/2} \nonumber \\
&\leq \lambda^{-1/4}\sqrt{2\hat{c}} \sum_{k=1}^n \exp(-\dot{c} (n-k)/2)  W_2(\calL(\bar{\theta}^{\lambda}_{kT}),\calL(\bar{\zeta}_{kT}^{\lambda,k-1})) \nonumber\\
&\quad + \lambda^{ 1/4}\sqrt{2\hat{c}} \sum_{k=1}^n \exp(-\dot{c} (n-k)/2) \left[1 + \left\lbrace \E[V_4(\bar{\theta}^{\lambda}_{kT})]\right\rbrace^{1/2} + \left\lbrace\E[V_4(\bar{\zeta}_{kT}^{\lambda,k-1})] \right\rbrace^{1/2}\right]\nonumber\\
& \leq  \sqrt{2\hat{c}}\lambda^{1/4}e^{-\min\{\dot{c},a/4\}n/2}n (e^{\min\{\dot{c},a/4\}/2}\bar{C}^{1/2}_{2,1}\mathbb{E}^{1/2}[V_2(\theta_0)]  +2\sqrt{2}\mathbb{E}^{1/2}[V_4(\theta_0)] ) \nonumber\\
&\quad + \sqrt{2\hat{c}}\lambda^{1/4}\frac{1 }{1 - \exp(-\dot{c}/2)}(\bar{C}^{1/2}_{2,2}+2\sqrt{2c_3}(\lambda_{\max}+a^{-1})^{1/2}+\sqrt{3}\mathrm{v}^{1/2}_4(\overline{M}_4)+\sqrt{15}) \nonumber\\
&\leq \lambda^{1/4}(e^{-\min\{\dot{c},a/4\}n/4}\bar{C}^*_{2,3}\mathbb{E}^{1/2}[V_4(\theta_0)] +\bar{C}^*_{2,4})\nonumber\\
&= \lambda^{1/4}(e^{-\dot{c}n/4}\bar{C}^*_{2,3}\mathbb{E}^{1/2}[V_4(\theta_0)] +\bar{C}^*_{2,4}),
\end{align*}
where
\begin{align}\label{mainthmcstar2}
\begin{split}
\bar{C}^*_{2,3}&: =\sqrt{2\hat{c}}\left(1+4/\dot{c}\right) (e^{a/8}\bar{C}^{1/2}_{2,1}  +2\sqrt{2})\\
\bar{C}^*_{2,4}&: = \frac{\sqrt{2\hat{c}}}{1 - \exp(-\dot{c}/2)}(\bar{C}^{1/2}_{2,2}+2\sqrt{2c_3}(\lambda_{\max}+a^{-1})^{1/2}+\sqrt{3}\mathrm{v}^{1/2}_4(\overline{M}_4)+\sqrt{15}),
\end{split}
\end{align}
with $\bar{C}_{2,1} $, $\bar{C}_{2,2}$ given in \eqref{barc2}, $\hat{c} $, $\dot{c}$ given in Lemma \ref{contractionconst}, $c_3$ given in \eqref{4thmomentconst}, and $\overline{M}_4$ given in Lemma \ref{lem:PreLimforDriftY}. This completes the proof.
\end{proof}

\begin{proof}[\textbf{Proof of Lemma~\ref{convergencepart3}}]\label{proof:convergencepart3}

By using Lemma \ref{convergencepart1} and \ref{convergencepart2}, one obtains
\begin{align*}
W_1(\mathcal{L}(\bar{\theta}^{\lambda}_t),\mathcal{L}(Z^{\lambda}_t))
&\leq W_1(\calL(\bar{\theta}^{\lambda}_t),\calL(\bar{\zeta}_t^{\lambda,n})) +W_1(\mathcal{L}(\bar{\zeta}_t^{\lambda,n}),\calL(Z_t^\lambda)) \\
& \leq  \sqrt{\lambda} (e^{-an /8}\bar{C}_{2,1}^{1/2}\mathbb{E}^{1/2}[V_2(\theta_0)]  +\bar{C}_{2,2}^{1/2}) \\
&\quad +\sqrt{\lambda}(e^{-\dot{c}n/2}\bar{C}_{2,3}\mathbb{E}[V_4(\theta_0)] +\bar{C}_{2,4})\\
&\leq  (\bar{C}_{2,1}^{1/2} +\bar{C}_{2,2}^{1/2}+ \bar{C}_{2,3}+\bar{C}_{2,4})\sqrt{\lambda}(e^{-\dot{c}n/2}\mathbb{E}[V_4(\theta_0)] +1).
\end{align*}
\end{proof}

\begin{proof}[\textbf{Proof of Lemma~\ref{contractionconst}}]\label{proof:contractionconst}

To obtain the contraction constant $\dot{c}$, we apply the arguments in the proof of \cite[Theorem 2.2]{eberle} to SDE \eqref{sde}. More precisely, we replace $h(r)$ in \cite[Eqn. (5.14)]{eberle} by
\begin{equation}\label{cchr}
h(r): = \frac{\beta}{4}\int_0^r s\kappa \, \rmd s +2Q(\epsilon)r,
\end{equation}
where $ \kappa = L_1\mathbb{E}[\eta(X_0)]$ and $Q(\epsilon)$ are given in \cite[Eqn. (2.24)]{eberle}, and replace \cite[Eqn. (2.25)]{eberle} by $(4 \tilde{c}(2) \epsilon)^{-1} \geq \frac{\beta}{2}\int_0^{R_1}\int_0^s \exp\left( \frac{\beta}{4}\int_r^s u\kappa \, \rmd u +2Q(\epsilon)(s-r)\right)\, \rmd r \,\rmd s$. Then, following the proof of \cite[Theorem 2.2]{eberle}, one can derive the expressions for $\dot{c}$: $\dot{c}:=\min\{\phi, \bar{c}(2), 4\tilde{c}(2) \epsilon\bar{c}(2)\}/2$, where $\bar{c}(2)= a/2$, $ \tilde{c}(2) = (3/2) a  \mathrm{v}_2(\overline{M}_2)$ with $\overline{M}_2$ given in Lemma~\ref{lem:PreLimforDriftY}, $\phi$ is given by
$
\phi^{-1} := \int_0^{R_2} \int_0^s\exp\left(\frac{\beta}{4}\int_r^s u \kappa \, \rmd u+2Q(\epsilon) (s-r)\right)\,\rmd r\, \rmd s
$ 
with $R_2$ given in \cite[Eqn. (2.29)]{eberle}, and $ \epsilon \in (0,1]$ is required to satisfy $\epsilon^{-1} \geq 2\beta \tilde{c}(2) \int_0^{R_1} \int_0^s\exp\left(\frac{\beta}{4}\int_r^s u \kappa \, \rmd u+2Q(\epsilon) (s-r)\right)\,\rmd r\, \rmd s$ with $R_1$ given in \cite[Eqn. (2.29)]{eberle}. To simply the expressions for $\phi$ and $\epsilon$, we follow the proof of \cite[Lemma~3.24]{nonconvex}, and thus \eqref{dotc}, \eqref{simplifiedphi}, \eqref{simplifiedep} can be obtained. 

To obtain an explicit expression for $\hat{c}$, one first notes that, by using \eqref{cchr}, \cite[Eqn. (5.4)]{eberle} becomes: for any $r \in [0,R_2]$, $r\exp(-\beta \kappa R_2^2/8-2Q(\epsilon) R_2) \leq \Phi(r) \leq 2f(r) \leq 2 \Phi(r) \leq 2r$. Then, in view of \cite[Eqn. (60)]{nonconvex}, and by applying the same arguments as in the proof of \cite[Lemma~3.24]{nonconvex}, one obtains $C_9 := C_{11}/C_{10} \leq \hat{c}: = 2(1+\overline{R}_2)\exp(\beta K_1\overline{R}_2^2/8+2\overline{R}_2)/\epsilon$, where $\overline{R}_2  =\bar{b}:=2\sqrt{4\tilde{c}(2)(1+\bar{c}(2))/\bar{c}(2)-1}$, $K_1 :=  L_1\E[\eta(X_0)]$, and $\epsilon$ is given in \eqref{simplifiedep}.
\end{proof}

\newpage

\section{Table of constant}\label{sec:consttable}
 
\begin{table}[h]
\tiny
\begin{center}
\begin{minipage}{\textwidth}
\caption{Analytic expressions of constants}
\label{tab:constantsexp}
\begin{tabular*}{1.05\textwidth}{@{\extracolsep{\fill}}ccl@{\extracolsep{\fill}}}
\toprule
\multicolumn{2}{c}{Constant}  &Full expression\\ 
\midrule
	\multirow{3}{*}{Lemma \ref{lem:PreLimforDriftY}}	&$\overline{M}_p$ 					& $\sqrt{1/3 + 4b/(3a) + 4d/(3a\beta) + 4(p-2)/(3a\beta)}$ \\\cline{2-3}
    																	&$\bar{c}(p)$							&$ap/4$  \\\cline{2-3}
      																	&$\tilde{c}(p)$							&$(3/4) a p \mathrm{v}_p(\overline{M}_p)$\\ \hline
	\multirow{4}{*}{Lemma \ref{convergencepart1}}		&$\bar{C}_{2,1} $						& $4 e^{4  L_1^2 \E\left[\eta^2(X_0)\right]}\Big(2\lambda_{\max}L_1^4\big( \E\left[\eta^2(X_0)\right] \big)^2  +8L_2^2\E[(\eta(X_0)+\eta(\E[X_0]))^2|X_0 - \E[X_0]|^2] \Big)$\\ \cline{2-3}
      																	&\multirow{3}{*}{$\tilde{C}_{2,2}$} & $4  e^{4  L_1^2 \E\left[\eta^2(X_0)\right]} \Big(  2\lambda_{\max}L_1^4  \big( \mathbb{E}[\eta^2(X_0)]\big)^2c_1(\lambda_{\max}+a^{-1})\Big.$\\ 
      																	&											&$+4\lambda_{\max}L_1^2L_2^2\mathbb{E}[\eta^2(X_0)]\mathbb{E}[\bar{\eta}^2(X_0)] +4\lambda_{\max} H_\star^2L_1^2 \E\left[\eta^2(X_0)\right]+2d\beta^{-1}L_1^2 \E\left[\eta^2(X_0)\right]$\\
      																	&											&$+\Big. 8L_2^2\E[(\eta(X_0)+\eta(\E[X_0]))^2|X_0 - \E[X_0]|^2](3\mathrm{v}_2(\overline{M}_2)+c_1(\lambda_{\max}+a^{-1})+1)\Big)$\\\hline
	\multirow{2}{*}{Lemma \ref{convergencepart2}}		& $\bar{C}_{2,3}$ 						&$\hat{c}\left(1+\frac{2}{\dot{c}}\right) (e^{a/4}\bar{C}_{2,1}  +12)$\\ \cline{2-3}
         																&$\bar{C}_{2,4}$ 						&  $\frac{\hat{c} }{1 - \exp(-\dot{c})}(\bar{C}_{2,2}+12c_3(\lambda_{\max}+a^{-1})+9\mathrm{v}_4(\overline{M}_4)+15)$ \\\hline
  \multirow{2}{*}{Corollary \ref{convergencepart2w2}}	&$\bar{C}^*_{2,3}$					& $\sqrt{2\hat{c}}\left(1+\frac{4}{\dot{c}}\right) (e^{a/8}\bar{C}^{1/2}_{2,1}  +2\sqrt{2})$ \\\cline{2-3}
        																 &$\bar{C}^*_{2,4}$ 					&$ \frac{\sqrt{2\hat{c}}}{1 - \exp(-\dot{c}/2)}(\bar{C}^{1/2}_{2,2}+2\sqrt{2c_3}(\lambda_{\max}+a^{-1})^{1/2}+\sqrt{3}\mathrm{v}^{1/2}_4(\overline{M}_4)+\sqrt{15})$\\ \hline
	\multirow{3}{*}{Theorem \ref{main}}					&$ C_1 $									& $2e^{\dot{c}/2} \left[ (\lambda_{\max}^{1/2}(\bar{C}_{2,1}^{1/2} +\bar{C}_{2,2}^{1/2}+ \bar{C}_{2,3}+\bar{C}_{2,4})+\hat{c})+\hat{c}\left(1+ \int_{\mathbb{R}^d}V_2(\theta)\pi_{\beta}(d\theta)\right)\right]$ \\\cline{2-3}
																	&$ C_2 $									&$\bar{C}_{2,1}^{1/2} +\bar{C}_{2,2}^{1/2}$\\\cline{2-3}
																	&$ C_3$									&$\bar{C}_{2,3}+\bar{C}_{2,4} $\\\hline
	\multirow{3}{*}{Corollary \ref{cw2}}					&$ C_4 $									&$2 \left(\lambda_{\max}^{1/2}(\bar{C}_{2,1}^{1/2} + \bar{C}_{2,2}^{1/2} )+\lambda_{\max}^{1/4}(\bar{C}^*_{2,3} + \bar{C}^*_{2,4})+\sqrt{2}\hat{c}^{1/2}\right)$   \\\cline{2-3}
																	&$ C_5 $									&$ \lambda_{\max}^{1/4}\bar{C}_{2,1}^{1/2} +\lambda_{\max}^{1/4}\bar{C}_{2,2}^{1/2}$\\\cline{2-3}
																	&$ C_6 $									&$\bar{C}^*_{2,3} + \bar{C}^*_{2,4}$\\\hline
	\multirow{3}{*}{Corollary \ref{eer}}						&$C^{\sharp}_1 $						& $ C_4\left(L_1 \mathbb{E}[\eta(X_0)](\mathbb{E}[|\theta_0|^2]+c_1(\lambda_{\max}+a^{-1})) +L_2\mathbb{E}[\bar{\eta}(X_0)]+H_{\star}\right)\E[|\theta_{0}|^{4}+1]$  \\\cline{2-3}
																	&$C^{\sharp}_2$						&$(C_5+C_6)\left(L_1 \mathbb{E}[\eta(X_0)](\mathbb{E}[|\theta_0|^2]+c_1(\lambda_{\max}+a^{-1})) +L_2\mathbb{E}[\bar{\eta}(X_0)]+H_{\star}\right)$\\\cline{2-3}
																	&$C^{\sharp}_3 $						&$\frac{d}{2\beta}\log\left(\frac{e L_1\mathbb{E}[\eta(X_0)]}{a}\left(\frac{b\beta}{d}+1\right)\right)$\\ 
          \bottomrule
    \end{tabular*}
\end{minipage}
\end{center}
\end{table}

\begin{table}[h!]
\tiny
\begin{center}
\begin{minipage}{1.0\textwidth}
\renewcommand{\thempfootnote}{\arabic{mpfootnote}}
\caption{ Constants in Lemma \ref{lem:moment_SGLD_2p} and Lemma \ref{contractionconst}, and their dependency on key parameters}\label{tab:constantskeydep}
\begin{tabular*}{1.05\textwidth}{@{\extracolsep{\fill}}cccc@{\extracolsep{\fill}}}
    \toprule
     \multicolumn{1}{c}{Constant}	&  \multicolumn{3}{c}{Key parameters}   \\
    \midrule
    \phantom{Constant} 			&$d$							&$\beta$							& Moments of $X_0$  \\\hline
         $c_1$ 							& $O(1+d/\beta)$ 			& $O(1+d/\beta)$				&$O (\E[(1+|X_0|)\eta(X_0)])$  \\\hline
          $c_3$							&$O(1+(d/\beta)^2)$ 		& $O(1+(d/\beta)^2)$			&$O (\E^{3/2}[(1+|X_0|)^4\eta^4(X_0)])$  \\\hline
    $\dot{c}$							& \multicolumn{3}{c}{$ \left(\frac{32\sqrt{\pi}(1+a^2)(1+\beta)}{a^2\sqrt{\beta}}\left(1+\frac{1}{\sqrt{L_1\E[\eta(X_0)]}}\right)e^{\left(8C^\star(a,b)(1+\beta L_1\E[\eta(X_0)])(1+\frac{d}{\beta})+\frac{16}{\beta L_1\E[\eta(X_0)] }\right)}\right)^{-1}$\footnotemark[1]} \\\hline
       $\hat{c}$							& \multicolumn{3}{c}{$O\left(\sqrt{\frac{\beta}{L_1\E[\eta(X_0)]}}(1+\frac{d}{\beta})^2e^{\left(12C^\star(a,b)(1+\beta L_1\E[\eta(X_0)])(1+\frac{d}{\beta})+\frac{16}{\beta L_1\E[\eta(X_0)] }\right)}\right)$\footnotemark[1]}  \\
       \bottomrule
    \end{tabular*}
    \footnotetext[1]{$C^\star(a,b)  := (1+2/a) (1+a+b)$.}
\end{minipage}
\end{center}
\end{table}

\newpage

\bibliographystyle{plainnat}
\bibliography{biblio}


\end{document}